\documentclass[12pt]{amsart}
\usepackage{amssymb,,amscd,latexsym}

\usepackage{amsmath, amsfonts, amsthm, amssymb} 
\usepackage{color}
\usepackage{hyperref}
\usepackage[all, cmtip]{xy}

\usepackage{tensor}

\newtheorem{theorem}{Theorem}
\newtheorem{lemma}[theorem]{Lemma}

\newtheorem{corollary}[theorem]{Corollary}

\newtheorem{conjecture}[theorem]{Conjecture}
\newtheorem{proposition}[theorem]{Proposition}

\def\se{{\subseteq}}

\newcommand{\R}{\ensuremath{\mathbb{R}}}
\newcommand{\Z}{\ensuremath{\mathbb{Z}}}
\newcommand{\Q}{\ensuremath{\mathbb{Q}}}
\newcommand{\N}{\ensuremath{\mathbb{N}}}

\newcommand{\F}{\ensuremath{\mathbb{F}}}

\newcommand{\bfg}{\ensuremath{\mathbf{G}}}

\newcommand{\mco}{\ensuremath{\mathcal{O}}}
\newcommand{\mcp}{\ensuremath{\mathcal{P}}}

\providecommand{\abs}[1]{\left\lvert#1\right\rvert}

\providecommand{\hcu}[1]{H_{cc}^{#1}}
\providecommand{\ccu}[1]{C_{cc}^{#1}}

\DeclareMathOperator{\Ext}{Ext}
\DeclareMathOperator{\Tor}{Tor}
\DeclareMathOperator{\cd}{cd}

\DeclareMathOperator{\Comm}{Comm}

\DeclareMathOperator{\Hom}{Hom}

\DeclareMathOperator{\SLtwo}{{\bf SL_2}}
\DeclareMathOperator{\Btwo}{{\bf B_2}}

\DeclareMathOperator{\Coind}{Coind}
\DeclareMathOperator{\supp}{supp}

\newcommand\suchthat{ \mathrel{}\middle| \mathrel{}}

\title[Semidualities]{Semidualities from products of trees}

\thanks{The authors gratefully acknowledge the support of the National Science Foundation.}

\author{Daniel Studenmund \& Kevin Wortman}

\begin{document}

\begin{abstract} 
  Let $K$ be a global function field of characteristic $p$, and let
  $\Gamma$ be a finite-index subgroup of an arithmetic group defined
  with respect to $K$ and such that any torsion element of $\Gamma$ is
  a $p$-torsion element.  We define semiduality groups, and we show
  that $\Gamma$ is a $\mathbb{Z}[1/p]$-semiduality group if $\Gamma$
  acts as a lattice on a product of trees. We also give other examples
  of semiduality groups, including lamplighter groups, Diestel-Leader
  groups, and countable sums of finite groups.
\end{abstract}

\maketitle

\section{Introduction}

\subsection{Arithmetic groups} Let $K$ be a global field (number or function field), and let $S$ be a nonempty set of finitely many inequivalent valuations of $K$ including each archimedean valuation. The ring $\mathcal{O}_S \subseteq K$ will denote  the corresponding ring of $S$-integers. For any $v \in S$, we let $K_v$ be the completion of $K$ with respect to $v$ so that $K_v$ is a locally compact field.

We let $\bf G$ be a noncommutative, absolutely almost simple algebraic $K$-group, so that ${\bf G}(\mathcal{O}_S)$ is a lattice, included diagonally, in the product of simple Lie groups $\prod _{v \in S} {\bf G}(K_v)$. For each $v\in S$, we let $X_v$ be the symmetric space or Euclidean building (depending on whether $K_v$ is an archimedean or nonarchimedean field) associated with ${\bf G}(K_v)$, and we let $X_S=\prod_{v \in S} X_v$ so that ${\bf G}(\mathcal{O}_S)$ acts on $X_S$ as a lattice.

We let 
 $$k({\bf G},S)=\sum_{v\in S}\text{rank}_{K_v}\bf G$$

If $\bf G$ is $K$-anisotropic --- that is, if  ${\bf G}(\mathcal{O}_S)$ acts cocompactly on $X_S$ --- then there is a finite-index subgroup of ${\bf G}(\mathcal{O}_S)$ that is a duality group, and if $K_v$ is an archimedean field --- that is, if $X_v$ is a symmetric space --- for all $v \in S$, then there is a finite-index subgroup of ${\bf G}(\mathcal{O}_S)$ that is a Poincar\'e duality group.

Borel-Serre \cite{B-S1} \cite{B-S2} showed that ${\bf G}(\mathcal{O}_S)$ is also a virtual duality group when $\bf G$ is $K$-isotropic, as long as $K$ is a number field. In particular, Borel-Serre construct a bordification of $X_S$, which we denote as $\widehat{X_S}$, on which ${\bf G}(K)$ acts and ${\bf G}(\mathcal{O}_S)$ acts properly and cocompactly, and such that the compactly supported cohomology groups $H_c^*(\widehat{X_S};\mathbb{Z})$ are nontrivial in some single dimension, $\ell({\bf G},S)$. The result is that any finite-index torsion-free subgroup of ${\bf G}(\mathcal{O}_S)$ is a duality group of dimension $\ell({\bf G},S)$
with dualizing module $H_c^{\ell({\bf G},S)}(\widehat{X_S};\mathbb{Z})$.

The purpose of this paper is to suggest a possible analogue of Borel-Serre for arithmetic groups ${\bf G}(\mathcal{O}_S)$ when $K$ is a global function field.

\subsection{Function field case} Throughout the remainder of this paper, $K$ denotes a global function field of characteristic $p$, and we suppose that $\bf G$ is $K$-isotropic --- that is, that  ${\bf G}(\mathcal{O}_S)$ does not act cocompactly on $X_S$.

Any finite-index subgroup of ${\bf G}(\mathcal{O}_S)$ contains torsion, so it cannot be a duality group, as duality groups have finite cohomological dimension. However, there are finite-index subgroups of ${\bf G}(\mathcal{O}_S)$ whose only torsion elements are $p$-elements and whose cohomological dimension over $\mathbb{Z}[1/p]$ is bounded above by $k({\bf G},S)$. We let $\Gamma$ denote such a subgroup.

The group $\Gamma$ still has an obstruction to being a $\mathbb{Z}[1/p]$-duality group. Indeed, it is not of type $FP_{k({\bf G},S)}$ over $\mathbb{Z}[1/p]$ (see Kropholler \cite{K}, Bux-Wortman \cite{B-W1}, Gandini \cite{Ga}, and Bux-K\"{o}hl-Witzel \cite{B-K-W}). However, $\Gamma$ is of type $FP_{k({\bf G},S)-1}$, and we conjecture that the discrepancy between type $FP_{k({\bf G},S)}$ and $FP_{k({\bf G},S)-1}$ is the only, and in some ways a minor, obstruction to $\Gamma$ being a $\mathbb{Z}[1/p]$-duality group. Before making this precise, we'll need a definition.

For a commutative ring $R$, we say that a group $\Lambda$ is an $R$-semiduality group of dimension $d$ if
\begin{quote}

({\emph{i}}) ${\rm cd}_R(\Lambda) \leq d$,

\noindent ({\emph{ii}}) $ \Lambda$ is of type $FP_{d-1}$ over $R$,

\noindent ({\emph{iii}}) $H^k(\Lambda ; R\Lambda)=0$ if $k\neq d$, and

\noindent ({\emph{iv}}) $H^d(\Lambda ; R\Lambda)$ is a flat $R$-module.
\end{quote}

In the above definition, $H^d(\Lambda ; R\Lambda)$ is called the \emph{dualizing module}, and if the ring $R$ and the group $\Lambda$ are understood, then we'll often denote the dualizing module simply as $D$.

In Section \ref{sec:homalg} of this paper we'll show the following consequence of a group being a semiduality group.

\begin{proposition}\label{p:43988hfq0f}
If $\Lambda$ is an $R$-semiduality group of dimension $d$, then for any $0
\leq n \leq d$ and any left $R \Lambda$-module $M$, there are natural homomorphisms of $R$-modules 
$$\varphi_n^M :H_n(\Lambda; D \otimes _R M) \rightarrow
H^{d-n}(\Lambda; M)$$
The $\varphi_n^M$ are compatible with the connecting homomorphisms in
the long exact homology and cohomology sequences associated to a short
exact sequence of coefficient modules, and if $\cdots \rightarrow Q_1
\rightarrow Q_0 \rightarrow M \rightarrow 0$ is a projective
resolution of $M$ by left $R\Lambda$-modules, then $\varphi_n^M$ is
injective if $Q_{n+1}$ and $Q_n$ are finitely generated, and
surjective if $Q_n$ and $Q_{n-1}$ are finitely generated. By convention, $Q_{-1}$ is always finitely generated.\end{proposition}

With the definition of semiduality and its immediate consequences listed above, we propose the following 

\begin{conjecture}\label{c:conj} Let $\mathcal{O}_S$ be a ring of $S$-integers in a global function field $K$ of characterstic $p$, and let $\bf G$ be a noncommutative, absolutely almost simple algebraic $K$-group.
If $\Gamma$ is a finite-index subgroup of ${\bf G}(\mathcal{O}_S)$ such that any torsion element of $\Gamma$ is a $p$-element, then $\Gamma$ is a $\mathbb{Z}[1/p]$-semiduality group of dimension $k({\bf G},S)$, and the dualizing module admits an action by ${\bf G}(K)$.
\end{conjecture}

Note that Bux-K\"ohl-Witzel \cite{B-K-W} shows that any $\Gamma$ as in Conjecture~\ref{c:conj} is of type $FP_{k({\bf G},S)-1}$ over $\mathbb{Z}[1/p]$, and it is well-known that $\cd_{\mathbb{Z}[1/p]}\Gamma \leq k({\bf G},S)$ since the dimension of $X_S$ equals $k({\bf G},S)$ (see Lemma~\ref{lem:cdlemma} below). Therefore, proving Conjecture~\ref{c:conj} would amount to proving that $H^{m}(\Gamma ; \mathbb{Z}[1/p]\Gamma)=0$ if $m < k({\bf G},S)$, and that $D=H^{k({\bf G},S)}(\Gamma ; \mathbb{Z}[1/p]\Gamma)$ is flat as a $\mathbb{Z}[1/p]$-module, the latter condition being equivalent to $D$ being torsion-free since $\mathbb{Z}[1/p]$ is a principal ideal domain.

Furthermore, we conjecture that  $D$ contains, and is an inverse limit of quotients of, $H_c^{k({\bf G},S)}(X_S;\mathbb{Z}[1/p])$. That is, we can view $D$ as an augmentation of $H_c^{k({\bf G},S)}(X_S;\mathbb{Z}[1/p])$. Thus, whereas Borel-Serre exhibits duality groups whose dualizing modules are cohomology groups of augmentations of the spaces on which arithmetic groups act, we conjecture that over function fields, arithmetic groups are semiduality groups whose dualizing modules are augmentations of cohomology groups of spaces on which the arithmetic groups act.

As an illustration, let $L$ be a field whose characteristic is not
equal to $p$.  Recall that $\cd_L(\Gamma) \leq k({\bf G}, S)$. By
Bux-K\"ohl-Witzel \cite{B-K-W}, $L$ is of type $FP_{k({\bf G}, S)-1}$
as a $\Z[1/p]\Gamma$-module. Therefore if Conjecture \ref{c:conj} is
true, then $ H^{1}(\Gamma; L)$ is a quotient of
$H_{k({\bf G},S)-1}(\Gamma; D)$, and if
 $ 2 \leq n\leq k({\bf G},S)$ then
\[
  H_{k({\bf G},S)-n}(\Gamma; D) \cong H^{n}(\Gamma; L)
\]
Note that the only dimension of $H^{*}(\Gamma; L)$ which semiduality
would not be able to help determine is dimension 0, but we know
$H^0(\Gamma; L)=L$.

\subsection{Main result}
What we prove in this paper is a first case of Conjecture~\ref{c:conj}. Namely

\begin{theorem}\label{t:mt} Conjecture~\ref{c:conj} is true if ${\rm{rank}}_{K_v}{\bf G}=1$ for all $v \in S$.
In particular, if ${\bf P}$ is a proper $K$-parabolic subgroup of $\bf G$, then there is an exact sequence of $\mathbb{Z}[1/p]{\bf G}(K)$-modules 
$$0 \longrightarrow  H_c^{k({\bf G},S)}(X_S;\mathbb{Z}[1/p]) \longrightarrow D \longrightarrow \bigoplus _{z \in ({\bf G/P})(K)} M_z \longrightarrow 0$$ where $M_z$ is an uncountable $\mathbb{Z}[1/p]$-module for each $z \in ({\bf G/P})(K)$, $M_z\cong M_w$ as $\mathbb{Z}[1/p]$-modules for any $z,w \in ({\bf G/P})(K)$, and $g(M_z)=M_{gz}$ for all $g \in {\bf G}(K)$ and $z \in ({\bf G/P})(K)$.
\end{theorem}

For example, ${\bf SL_2}(\mathbb{F}_p[t])$ is a semiduality group of dimension 1, ${\bf SL_2}(\mathbb{F}_p[t,t^{-1}])$ is a semiduality group of dimension 2, and ${\bf SL_2}(\mathcal{O}_S)$ is a semiduality group of dimension $|S|$ whose dualizing module incorporates the action of ${\bf SL_2}(K)$ on $\mathbb{P}^1(K)$.

Our proof of Theorem~\ref{t:mt} is geometric. That is, we will use strongly that, under the hypotheses of Theorem~\ref{t:mt}, $X_S$ is a product of trees.

\subsection{Solvable groups} Let $\bf B_2$ be the group of upper
triangular matrices of determinant 1.  Thus,
${\bf B_2}(\mathbb{F}_p[t])$ is commensurable to $ \mathbb{F}_p[t]$,
and ${\bf B_2}(\mathbb{F}_p[t,t^{-1}])$ is commensurable to the
lamplighter group $\mathbb{F}_p \wr \mathbb{Z}$. This paper will also
show

\begin{theorem} ${\bf B_2}(\mathcal{O}_S)$ is virtually a
  $\Z[1/p]$-semiduality group of dimension $|S|$.
\end{theorem}

Thus, as the discrete group Solv is known to be a Poincar\'e duality group, and as the solvable Baumslag-Solitar
groups are known to be duality groups, the lamplighter groups with prime order
cyclic base are semiduality groups. Notice that the three groups from the previous sentence are commensurable respectively to ${\bf B_2}(\mathbb{Z}[\sqrt{2}])$,
${\bf B_2}(\mathbb{Z}[1/p])$, and ${\bf B_2}(\mathbb{F}_p[t,t^{-1}])$.

We also show that certain generalizations of $\Btwo(\F_p[t])$ and
$\Btwo(\F_p[t, t^{-1}])$ are semiduality groups, namely
countable sums of finite groups and Diestel-Leader groups,
respectively.

\subsection{Outline of proof}
In Section \ref{sec:homalg} we'll prove
Proposition~\ref{p:43988hfq0f}. In Section \ref{sec:trans_to_topology}
we'll show how the cohomology of a discrete group with group ring
coefficients can be, in some cases, interpreted from the topology of a
contractible space on which it acts properly, and perhaps
noncocompactly. In Section \ref{s:red} we'll detail how the groups
${\bf G}(\mathcal{O}_S)$ from Theorem~\ref{t:mt} act cocompactly on
the complement of a pairwise disjoint collection of horoballs in a
product of trees, and in Section \ref{sec:horoball_comp} we'll show
that such a complement has trivial compactly supported cohomology in
dimension $d-1$, where $d$ is the number of factors in the
product. Section \ref{sec:horosphere} shows that $\varprojlim ^1$ of
the compactly supported cohomology of a nested sequence of regular
horospheres in a product of trees is torsion-free in dimension $d-1$,
and the final section of this paper, Section \ref{sec:examples}, will combine the
ingredients collected in earlier sections to prove that certain groups
are semiduality groups, including a proof of Theorem~\ref{t:mt}.

\subsection{Acknowledgements}
We thank Gopal Prasad for suggesting to us the problem of looking for
an analogue to Borel-Serre for arithmetic groups defined with respect
to a global function field. We thank Mladen Bestvina for advice at
different stages of this project. We thank Ross Geoghegan, Peter
Kropholler, Tim Riley, and Thomas Weigel for helpful conversations.
We are also grateful to a referee of this paper who provided us with the statement and proof of Proposition~\ref{p:refchar}. We also thank Srikanth Iyengar for suggesting to us the proof of Lemma~\ref{l:sri} given below.

\section{Homological Algebra} \label{sec:homalg}

This section is divided in  two parts. First, we'll prove Proposition~\ref{p:43988hfq0f} from the introduction. Second, we'll prove an equivalent characterization of semiduality groups in the form of Proposition~\ref{p:refchar}. While Proposition~\ref{p:refchar} is not directly applied in this paper to other results, it seems to be of independent interest. Both the statement and proof of Proposition~\ref{p:refchar} were provided to us by an anonymous referee to whom we are grateful.

In this section we let $R$ be a commutative ring.

\subsection{Proof of Proposition~\ref{p:43988hfq0f}}\label{s:2.1}
We'll prove Proposition~\ref{p:43988hfq0f} in four steps. First, we'll define $\varphi_n^M$. Second,  we'll show how the injectivity and surjectivity of $\varphi_n^M$ can be deduced from the finiteness properties of $M$. Third, we'll demonstrate the required naturality properties of $\varphi_n^M$. Last, and not until the final sentence of Section~\ref{s:2.1}, we'll invoke the assumption from Proposition~\ref{p:43988hfq0f} that $H^d(\Gamma ; R\Gamma)$ is a flat $R$-module.

To define $\varphi_n^M$, let  $\Gamma$ be a group of type $FP_{d-1}$ over $R$ with  $\cd_R\Gamma \leq d$.
Then there is a projective resolution of the trivial left $R\Gamma$-module $R$ by left $R\Gamma$-modules
$$0 \rightarrow P_d \rightarrow P_{d-1} 
\rightarrow \cdots \rightarrow P_1 \rightarrow P_0 \rightarrow R \rightarrow 0$$  where $P_i$ is finitely generated if $i\leq d-1$. 

We let $P^*_i=\Hom_{R\Gamma}(P_i,R\Gamma)$ and $D=H^d(\Gamma ;
R\Gamma) $. Assuming that $H^k(\Gamma ;
R\Gamma) = 0$ if $k \neq d$, we have the below exact sequence of right $R\Gamma$-modules

$$0 \rightarrow P_0^* \rightarrow P_{1}^* 
\rightarrow \cdots \rightarrow P_{d-1}^* \rightarrow P_d^* \rightarrow D \rightarrow 0$$

Let $\cdots \rightarrow A_1 \rightarrow A_0 \rightarrow D \rightarrow 0$ be a projective resolution of $D$ by right $R\Gamma$-modules. Let $h:A_\bullet \rightarrow P_{d-\bullet}^*$ be a chain map over the identity map on $D$. Thus, for any left $R\Gamma$-module $M$, there are chain maps $$A_\bullet \otimes _{R\Gamma} M \rightarrow  P_{d-\bullet}^* \otimes _{R\Gamma} M \rightarrow \Hom_{R\Gamma}(P_{d-\bullet},M)$$
that induce natural homomorphisms

\xymatrix{  & \Tor_n^{R\Gamma}(D,M) \ar[r]^-{\kappa^M_n}  & H_n( P_{d-\bullet}^* \otimes _{R\Gamma} M ) \ar[r]^-{\nu^M_n}  & H^{d-n}(\Gamma;M)   }

We define $\varphi^M_n=\nu^M_n \circ \kappa^M_n$.

To deduce when $\varphi^M_n$ is injective or surjective, we'll deduce when those properties are satisfied by $ \kappa^M_n$ and $\nu^M_n$ separately, beginning  with $ \kappa^M_n$. But before that, we'll need the following

\begin{lemma} \label{l:higher_Qs_fg}
Let $P$ be a projective, left $R\Gamma$-module, and $P^*$ its dual. 
If $\cdots \rightarrow Q_1 \rightarrow Q_0 \rightarrow M \rightarrow 0$ is a projective resolution of a left $R\Gamma$-module $M$, and if $Q_{n+1}$ and $Q_{n}$ are finitely generated, then $\Tor_{n}^{R\Gamma} (P^* ,M)=0$.
\end{lemma}

\begin{proof}
The following diagram commutes 
\[
 \begin{CD} 
   {P}^* \otimes _{R\Gamma} Q_{n+1} @>>>  {P}^* \otimes _{R\Gamma} Q_{n}    @>>>  {P}^* \otimes _{R\Gamma} Q_{n-1} 
    \\ @VVV @VVV @VVV \\
   \Hom_{R\Gamma}(P,Q_{n+1}) @>>>  \Hom_{R\Gamma}(P,Q_n) @>>>  \Hom_{R\Gamma}(P,Q_{n-1})
 \end{CD}
  \]
Since  $Q_{n+1}$ and $Q_{n}$ are finitely generated projective, they
are finitely presented, so 
the two vertical maps on the left are isomorphisms. It follows that the
homology of the top row injects into
the homology of the bottom row. That is, 
$\Tor_{n}^{R\Gamma} (P^* ,M)$ injects into $0$, since $P$ is projective, and thus  
$\Hom(P,-)$ is exact.
\end{proof}

Now we are prepared to address the injectivity and surjectivity of the homomorphisms $\kappa_n^M:\Tor_{n}^{R\Gamma}(D,M) \rightarrow   H_n( P_{d-\bullet}^* \otimes _{R\Gamma} M )   $.

\begin{lemma}\label{l:sri}
Let $\cdots \rightarrow Q_1 \rightarrow Q_0 \rightarrow M \rightarrow 0$ be a projective resolution of a left $R\Gamma$-module $M$. Then $\kappa_0^M$ is an isomorphism, $\kappa_1^M$ is surjective,  $\kappa_n^M$ is  injective if $Q_{n+1}$ and $Q_{n}$ are
  finitely generated, and $\kappa_n^M$ is surjective if $Q_{n}$ and $Q_{n-1}$ are
  finitely generated.
\end{lemma}

\begin{proof}
Recall that $A_\bullet \rightarrow D$ is a projective resolution of $D$ with a chain map $h: A_\bullet \rightarrow P_{d-\bullet}^*$ over the identity on $D$, so that $h$ is a quasi-isomorphism. We let  $P^*_{d-\geq 1}$ be the truncated complex of $P^*_{d-\bullet}$, so that the bottom row of the below diagram of chain maps is exact.

\xymatrix{   &   & A_\bullet   \ar[d]^{h} &   &   \\
0 \ar[r]  & P^*_d  \ar[r] &  P^*_{d-\bullet}   \ar[r]  &  P^*_{d-\geq 1}  \ar[r] &  0  }

Because each $Q_n$ is projective, the below diagram of bicomplexes has an exact row and  ${(h\otimes id)_*}$ is a quasi-isomorphism.

\xymatrix{  &  & A_\bullet \otimes _{R\Gamma} Q_{\bullet}  \ar[d]^-{(h\otimes id)_*} &   &   \\
0 \ar[r]  & P^*_d \otimes _{R\Gamma} Q_{\bullet} \ar[r] &  P^*_{d-\bullet} \otimes _{R\Gamma} Q_{\bullet}  \ar[r]  &  P^*_{d-\geq 1} \otimes _{R\Gamma} Q_{\bullet} \ar[r] &  0  }

Passing to homology, we have the following diagram whose row is exact and whose vertical map, ${(h\otimes id)_*}$, is an isomorphism.

\xymatrix{     & \Tor_n^{R\Gamma}(D,M)  \ar[d]^-{(h\otimes id)_*} &  &   \\
\Tor_n^{R\Gamma}(P^*_d,M) \ar[r]   & H_n( P^*_{d-\bullet} \otimes _{R\Gamma} Q_{\bullet})  \ar[r] ^-{\iota _*} & H_n( P^*_{d-\geq 1} \otimes _{R\Gamma} Q_{\bullet})  \ar[r]^-\partial & 
\Tor_{n-1}^{R\Gamma}(P^*_d,M)   }

Notice that the domain of $\partial$ is $H_n( P^*_{d-\geq 1} \otimes _{R\Gamma} M)$ since $P^*_{d-\geq 1}$ is a deleted projective resolution.

For $n\geq 2$, we have $H_n( P^*_{d-\geq 1} \otimes _{R\Gamma} M)=H_n( P^*_{d-\bullet} \otimes _{R\Gamma} M)$ and $\kappa^M_n=\iota _* \circ (h\otimes id)_*$. Thus, when $n \geq 2$, the lemma follows from Lemma~\ref{l:higher_Qs_fg}.

For $n=1$, we have that $\Tor_{n-1}^{R\Gamma}(P^*_d,M) =P^*_d\otimes_{R\Gamma}M$ and thus the image of $\iota_*$, which is the kernel of $\partial$, is $H_1( P^*_{d-\bullet} \otimes _{R\Gamma} M)$.

For $n=0$, notice that $\kappa_0^M$ is the identity map on $D\otimes_{R\Gamma}M$.
\end{proof}

For $\nu_n^M:H_n( P_{d-\bullet}^* \otimes _{R\Gamma} M ) \rightarrow  H^{d-n}(\Gamma;M)  $ we have the following

\begin{lemma}\label{l:adendum}
Let $M$ be a left $R\Gamma$-module $M$. Then $\nu_1^M$ is injective, and $\nu_n^M$ is bijective if $n\geq 2$. If $M$ is finitely presented, then $\nu_0^M$ and $\nu_1^M$ are bijective. If $M$ is finitely generated, then $\nu_0^M$ is surjective.
\end{lemma}

 \begin{proof}
 If $n \geq 2$, then we see that the vertical maps in the below commutative diagram are isomorphisms since $P_{d-i}$ is finitely generated and projective if $i \geq 1$.
 \[
 \begin{CD} 
{P}^*_{d-(n+1)} \otimes _{R\Gamma} M @>>>   {P}^*_{d-n} \otimes _{R\Gamma} M @>>>  {P}^*_{d-(n-1)} \otimes _{R\Gamma} M  \\ 
@VVV@VVV @VVV \\
  \Hom_{R\Gamma}(P_{d-(n+1)},M) @>>>   \Hom_{R\Gamma}(P_{d-n},M) @>>> \Hom_{R\Gamma}(P_{d-(n-1)},M) 
 \end{CD}
  \]
 Therefore, $\nu_n^M$ is an isomorphism if $n \geq 2$.

 For $n=1$, observe the below commutative diagram 
 \[
 \begin{CD} 
{P}^*_{d-2} \otimes _{R\Gamma} M @>>>   {P}^*_{d-1} \otimes _{R\Gamma} M @>>>  {P}^*_{d} \otimes _{R\Gamma} M  \\ 
@VVV@VVV @VVV \\
  \Hom_{R\Gamma}(P_{d-2},M) @>>>   \Hom_{R\Gamma}(P_{d-1},M) @>>> \Hom_{R\Gamma}(P_{d},M) 
 \end{CD}
  \]
Since $P_{d-2}$ and $P_{d-1}$ are finitely generated and projective, the two vertical maps on the left are isomorphisms so $\nu_1^M$ is injective. If, in addition,  $M$ is finitely presented, then the vertical map on the right is bijective, so  $\nu_1^M$ is bijective. 

 As for $\nu_0^M$,
the sequence 
 $$\Hom_{R\Gamma}(P_{d-1},M) \rightarrow \Hom_{R\Gamma}(P_{d},M) \rightarrow H^d(\Gamma;M) \rightarrow 0$$
 is exact by the definition of $H^{d}(\Gamma;M)$. Using the above for $M$ and $M=R\Gamma$, and using that tensor product is right exact, we see that the rows of the commutative diagram
 \[
 \begin{CD} 
 {P}^*_{d-1} \otimes _{R\Gamma} M @>>>  {P}^*_{d} \otimes _{R\Gamma} M @>>>
D \otimes _{R\Gamma} M @>>> 0 \\ 
@VVV@VVV @VVV  \\
   \Hom_{R\Gamma}(P_{d-1},M) @>>> \Hom_{R\Gamma}(P_{d},M) @>>> H^d(\Gamma;M) @>>> 0
 \end{CD}
  \]
  are exact.
  
The vertical map on the left is an isomorphism since $P_{d-1}$ is finitely generated and projective. Thus, $\nu_0^M$ is surjective (resp.\
bijective) if the second vertical map from the left is by the
5-Lemma. Now note that the second vertical map on the left is
surjective if $M$ is finitely generated, and bijective if $M$ is
finitely presented. 
\end{proof}

We now have

\begin{proposition}\label{p:injasurj}
  Suppose $\Gamma$ is a group of type $FP_{d-1}$ over $R$, that $\cd_R(\Gamma) \leq d$, and that $H^k(\Gamma ; R\Gamma)$ if $k \neq d$. Let $D= H^d(\Gamma ; R\Gamma)$. 
  
 If $\cdots \rightarrow Q_1 \rightarrow Q_0 \rightarrow M \rightarrow 0$ is a projective resolution of a left $R\Gamma$-module $M$,  then $$\varphi^M_n : \Tor^{R\Gamma}_n(D,M) \rightarrow
  H^{d-n}(\Gamma; M)$$ is
 injective if $Q_{n+1}$ and $Q_{n}$ are
  finitely generated, and surjective if $Q_{n}$ and $Q_{n-1}$ are
  finitely generated. \end{proposition}

\begin{proof}
Combine Lemmas~\ref{l:sri} and~\ref{l:adendum}.
\end{proof}

Having established conditions for the injectivity and surjectivity of the $\varphi_n^M$, we turn to naturality properties of these homomorphisms,  beginning with

\begin{lemma}
For a left $R\Gamma$-module $N$ and an exact sequence of left $R\Gamma$-modules $0 \rightarrow M \rightarrow M' \rightarrow M''
 \rightarrow 0$, the composition  $$\Tor_1^{R\Gamma}(N^*,M'')  \rightarrow N^* \otimes _{R\Gamma} M \rightarrow \Hom_{R\Gamma}(N,M)$$ is $0$, 
where $N^*= \Hom_{R\Gamma}(N,R\Gamma)$.
\end{lemma}

\begin{proof}
The proof follows from the below commutative diagram with exact rows
\[
 \begin{CD} 
\Tor_1^{R\Gamma}(N^*,M'') @>>>   N^* \otimes _{R\Gamma} M @>>>  N^* \otimes _{R\Gamma} M'  \\ 
& & @VVV @VVV \\
0  @>>>   \Hom_{R\Gamma}(N,M) @>>> \Hom_{R\Gamma}(N,M') 
 \end{CD}
  \]
\end{proof}

We will use the previous lemma to construct a diagram of chain complexes in our proof of the following 

\begin{proposition}\label{p:natural7}
  Suppose $\Gamma$ is a group of type $FP_{d-1}$ over $R$, that $\cd_R(\Gamma) \leq d$, and that $H^k(\Gamma ; R\Gamma)$ if $k \neq d$. Let $D= H^d(\Gamma ; R\Gamma)$. 
  
  For any left $R\Gamma$-module $M$, the $R$-module homomorphisms
  $$\varphi^M_n : \Tor^{R\Gamma}_n(D,M) \rightarrow
  H^{d-n}(\Gamma; M)$$ are natural, and they are compatible with the connecting homomorphisms in the long exact homology and cohomology sequences associated to a short exact sequence of coefficient modules.
\end{proposition}

\begin{proof}
That the $\varphi^M_n$ are natural follows from their definition, since homology is a functor.

To show that the $\varphi^M_n$ are compatible with connecting homomorphisms, we let $0\rightarrow M \rightarrow M' \rightarrow M'' \rightarrow 0 $ be a short exact sequence of left $R\Gamma$-modules, and we let $J$ be the image of $\Tor_1^{R\Gamma}(P_d^*,M'')  \rightarrow P_d^* \otimes _{R\Gamma} M $, so that we have an exact sequence $$ 0 \rightarrow (P_d^* \otimes_{R\Gamma} M)/J   \rightarrow  P_d^* \otimes_{R\Gamma} M'   \rightarrow  P_d^* \otimes_{R\Gamma}  M''  \rightarrow  0$$
Furthermore, if $i \geq 1$, then $P_{d-i}$ is finitely generated so $P^*_{d-i}$ is projective, and thus,  if we let $(P_{d-i}^* \otimes_{R\Gamma} M)_J = P_{d-i}^* \otimes_{R\Gamma} M$ for $i \geq 1$ and $(P_{d}^* \otimes_{R\Gamma} M)_J = (P_{d}^* \otimes_{R\Gamma} M)/J$, then we have a short exact sequence of chain complexes
$$ 0 \rightarrow     (P_{d-\bullet}^* \otimes_{R\Gamma} M)_J     \rightarrow   P_{d-\bullet}^* \otimes_{R\Gamma} M' \rightarrow   P_{d-\bullet}^* \otimes_{R\Gamma} M'' \rightarrow  0$$

By the previous lemma, $J$ is contained in the kernel of  $P_d^* \otimes _{R\Gamma} M \rightarrow \Hom_{R\Gamma}(P_d,M)$ so the composite
$$  A_0 \otimes _{R\Gamma} M  \rightarrow   P_{d}^* \otimes_{R\Gamma} M   \rightarrow \Hom_{R\Gamma}(P_{d},M)   $$ factors through
$    (P_{d-\bullet}^* \otimes_{R\Gamma} M)/J  $.
Thus, we have a commutative diagram of chain complexes with exact columns

\[
 \begin{CD} 
 0   & & 0   & & 0 \\
@VVV @VVV @VVV \\
  A_\bullet \otimes _{R\Gamma} M  @>>>   (P_{d-\bullet}^* \otimes_{R\Gamma} M)_J   @>>>  \Hom_{R\Gamma}(P_{d-\bullet},M)   \\ 
 @VVV @VVV @VVV \\
 A_\bullet \otimes _{R\Gamma} M' @>>>  P_{d-\bullet}^* \otimes_{R\Gamma} M'  @>>>  \Hom_{R\Gamma}(P_{d-\bullet},M')   \\ 
@VVV @VVV @VVV \\
A_\bullet \otimes _{R\Gamma} M''    @>>>    P_{d-\bullet}^* \otimes_{R\Gamma} M''  @>>>   \Hom_{R\Gamma}(P_{d-\bullet},M'')    \\ 
 @VVV @VVV @VVV \\
 0   & &  0 & & 0 \\
 \\
 \end{CD}
  \]

Therefore, for all $n$, we have the following commutative diagram on homology, whose vertical maps are connecting homomorphisms

\[
 \begin{CD} 
  \Tor^{R\Gamma}_n(D, M'')  @>>>   H_{n}(P_{d-\bullet}^* \otimes_{R\Gamma} M'' )  @>>>  H^{d-n}(\Gamma ; M'')   \\ 
 @VVV @VVV @VVV \\
   \Tor^{R\Gamma}_{n-1}(D, M) @>>>  H_{n-1}((P_{d-\bullet}^* \otimes_{R\Gamma} M)_J )   @>>>  H^{d-n+1}(\Gamma ; M )     \\ 
 \\
 \end{CD}
  \]

The top row of the above diagram is $\varphi_n^{M''}$. The bottom row is $\varphi_{n-1}^{M}$ as can be seen by noting that the below commutative diagram of chain maps

\xymatrix{
& & P_{d-\bullet}^* \otimes_{R\Gamma} M \ar[d] \ar[rd] \\
& A_\bullet \otimes _{R\Gamma} M  \ar[ru] \ar[r]   & (P_{d-\bullet}^* \otimes_{R\Gamma} M)_J  \ar[r]  & \Hom_{R\Gamma}(P_{d-\bullet},M)  }

\noindent yields the commutative diagram on homology

\xymatrix{
& & H_{n-1}(P_{d-\bullet}^* \otimes_{R\Gamma} M) \ar[d] \ar[rd] \\
& \Tor^{R\Gamma}_{n-1}(D, M)  \ar[ru] \ar[r]   & H_{n-1}((P_{d-\bullet}^* \otimes_{R\Gamma} M)_J ) \ar[r]  & H^{d-n+1}(\Gamma ; M ) \\
  }  \noindent and noticing that the top row of the above diagram is $\varphi_{n-1}^{M}$ while the bottom row coincides with the bottom row of the preceding commutative rectangle.
  \end{proof}

Last, if we assume that $D$ is flat as an $R$-module, so that $\Tor^{R\Gamma}_n(D,M)\cong H_n(\Gamma;D\otimes _R M)$, then Proposition~\ref{p:43988hfq0f} is Propositions~\ref{p:injasurj} and~\ref{p:natural7}.

\subsection{Alternative characterization of semiduality}
In what remains of this section, we'll prove the following result, given to us by a referee.

\begin{proposition}\label{p:refchar}
Let $\Gamma$ be a group, and suppose $\cd_R\Gamma =d$.
 The following are equivalent:

\bigskip

\noindent {\rm(}i{\rm)} $\Gamma$ is of type $FP_{d-1}$ over $R$, and $H^k(\Gamma;R\Gamma)=0$ if $k \neq d$.

\bigskip

\noindent {\rm(}ii{\rm)} There is a right $R\Gamma$-module $C$ such that the natural map $C \rightarrow C^{**}$ is an injection, and such that if $0\leq n \leq d-2$ then there are natural isomorphisms $$\xi_n: \Tor^{R\Gamma}_{d-1-n}(C,M)\rightarrow H^{n}(\Gamma; M)$$ and  $$\psi _n:H_n(\Gamma;N) \rightarrow \Ext^{d-1-n}_{R\Gamma}(C,N)$$ for left $R\Gamma$-modules $M$ and right $R\Gamma$-modules $N$, and such that there is an exact sequence $$0 \rightarrow H^{d-1} (\Gamma ; M) \rightarrow C \otimes _{R\Gamma} M \rightarrow \Hom_{R\Gamma}(C^*,M)$$ 

\bigskip

\noindent {\rm(}iii{\rm)} All of the conditions in {\rm(}ii{\rm)} hold. In addition, $C^{**}/C\cong H^d(\Gamma;R\Gamma)$, the exact sequence from {\rm(}ii{\rm)} extends to an exact sequence
$$0 \rightarrow H^{d-1} (\Gamma ; M) \rightarrow C \otimes _{R\Gamma} M \rightarrow \Hom_{R\Gamma}(C^*,M)\rightarrow H^d(\Gamma,M) \rightarrow 0$$ and there is an exact sequence
$$0\rightarrow H_d(\Gamma;N) \rightarrow B \otimes _{R\Gamma} C^* \rightarrow \Hom_{R\Gamma}(C,N) \rightarrow H_{d-1}(\Gamma ; N) \rightarrow 0$$

\end{proposition}

\begin{proof}
Since  ($iii$) is a strictly larger collection of conditions than ($ii$),  ($iii$) implies ($ii$), so in this proof we'll show that ($i$) implies ($iii$) and that ($ii$) implies ($i$). We begin with ($i$) implies ($iii$).

As in Section~\ref{s:2.1}, there is a projective resolution of the trivial left $R\Gamma$-module $R$
$$0 \rightarrow P_d \rightarrow P_{d-1} 
\rightarrow \cdots \rightarrow P_1 \rightarrow P_0 \rightarrow R \rightarrow 0$$  where $P_i$ is finitely generated if $i\leq d-1$, and we let $C$ be the image of $P_{d-1}^* \rightarrow P_{d}^*$. The map  $\xi_n$ is what we had previously named $\nu_n^M$ in the proof of Proposition~\ref{p:natural7}. The proof of the existence of the $\psi_n$ is similar. Indeed, since the $P_i$ are finitely generated projective if $i\leq d-1$, the vertical maps of the commutative diagram below are isomorphisms.

\[
 \begin{CD} 
N \otimes _{R\Gamma} P_{d-1} @>>> \cdots @>>> N \otimes _{R\Gamma} P_{0}  @>>> 0
    \\    @VVV & & @VVV \\
  \Hom_{R\Gamma}(P_{d-1}^*,N)  @>>> \cdots @>>> \Hom_{R\Gamma}(P_{0}^*,N) @>>> 0
 \end{CD}
  \]
The isomorphisms $\psi_n$ are the isomorphisms from the homologies of the two rows.

Now note that because $\Hom_{R\Gamma}(-,R\Gamma)$ is left-exact, the bottom row is exact in the diagram

\[
 \begin{CD} 
 0 @>>>  {P}_d  @>>> {P}_{d-1}  @>>> {P}_{d-2}
    \\  & & & & @VVV @VVV \\
 0 @>>>  C^*  @>>> {P}_{d-1}^{**}  @>>> {P}_{d-2}^{**}
 \end{CD}
  \]
The top row is exact by stipulation, and since $P_{d-1}$ and $P_{d-2}$ are finitely generated projective, the natural vertical maps are isomorphisms. It follows that $C^* \cong P_d$.

Applying the left-exact $\Hom_{R\Gamma}(-,N)$ to  $ P_{d-2}^* \rightarrow P_{d-1}^* \rightarrow C \rightarrow 0$ yields the exact sequence 
$$0\rightarrow \Hom_{R\Gamma}(C,N) \rightarrow \Hom_{R\Gamma}(
P_{d-1}^*,N) \rightarrow \Hom_{R\Gamma}(P_{d-2}^*,N) $$
The isomorphism $N \otimes _{R\Gamma} P_i \cong \Hom_{R\Gamma}(P_i^*,N)$ when $i=d-1,d-2$ implies $\Hom_{R\Gamma}(C,N) $ is the kernel of $N \otimes _{R\Gamma} P_{d-1}  \rightarrow N \otimes _{R\Gamma} P_{d-2}$. Therefore we have exact sequences 
$$N \otimes _{R\Gamma} P_d \rightarrow  \Hom_{R\Gamma}(C,N) \rightarrow H_{d-1}(\Gamma;N) \rightarrow 0$$ and
$$0 \rightarrow H_d(\Gamma;N)\rightarrow  N \otimes _{R\Gamma} P_d \rightarrow \Hom_{R\Gamma}(C,N)$$
Combining these two sequences and replacing $P_d$ with $C^*$ yields the exact sequence
$$0\rightarrow H_d(\Gamma;N) \rightarrow N \otimes _{R\Gamma} C^* \rightarrow \Hom_{R\Gamma}(C,N) \rightarrow H_{d-1}(\Gamma ; N) \rightarrow 0$$

Similarly, we apply the right-exact $-\otimes_{R\Gamma} M$ to $ P_{d-2}^* \rightarrow P_{d-1}^* \rightarrow C \rightarrow 0$ for the exact sequence
 $ P_{d-2}^*\otimes_{R\Gamma} M \rightarrow P_{d-1}^*\otimes_{R\Gamma} M \rightarrow C\otimes_{R\Gamma} M \rightarrow 0$. When $i=d-1,d-2$, we have $P_i^*\otimes_{R\Gamma} M \cong \Hom_{R\Gamma}(P_i,M)$. Hence, $C\otimes_{R\Gamma} M$ is the cokernel of $\Hom_{R\Gamma}(P_{d-2},M) \rightarrow \Hom_{R\Gamma}(P_{d-1},M)$ so we have the exact sequences
$$0 \rightarrow H^{d-1} (\Gamma ; M) \rightarrow C \otimes _{R\Gamma} M \rightarrow \Hom_{R\Gamma}(P_d,M)$$
and
$$C \otimes _{R\Gamma} M \rightarrow \Hom_{R\Gamma}(P_d,M)\rightarrow H^d(\Gamma,M) \rightarrow 0$$ 
Therefore, the below is exact 
$$0 \rightarrow H^{d-1} (\Gamma ; M) \rightarrow C \otimes _{R\Gamma} M \rightarrow \Hom_{R\Gamma}(C^*,M)\rightarrow H^d(\Gamma,M) \rightarrow 0$$

The final conditions of ($iii$) to be checked are that $0 \rightarrow C \rightarrow C^{**} \rightarrow H^d(\Gamma,R\Gamma) \rightarrow 0$ is exact. For this, let $M=R\Gamma$ in the sequence immediately preceding this paragraph. Our proof that ($i$) implies ($iii$) is complete.

To show that ($ii$) implies ($i$), note that the isomorphisms $\psi _n:H_n(\Gamma;N) \rightarrow \Ext^{d-n+1}_{R\Gamma}(C,N)$  for $0\leq n \leq d-2$ and any right $R\Gamma$-module $N$ show that $H_n(\Gamma;-)$ commutes with direct products when $0\leq n \leq d-2$. In particular, $H_n(\Gamma;\prod R\Gamma) \cong \prod H_n(\Gamma;R\Gamma)$ for $0\leq n \leq d-2$. That is, for $n=0$ we have $R \otimes_{R\Gamma} (\prod R\Gamma) \cong \prod R$ and for $1\leq n \leq d-2$ we have $ \Tor^{R\Gamma}_n(R;\prod R\Gamma) =0$, so that $\Gamma$ is of type $FP_{d-1}$ by Lemma 1.1 and Proposition 1.2 of \cite{B-E}.

For $0\leq n \leq d-2$, the existence of isomorphisms $\xi_n: \Tor^{R\Gamma}_{d-1-n}(C,R\Gamma)\rightarrow H^{n}(\Gamma; R\Gamma)$ shows that $H^{n}(\Gamma; R\Gamma)=0$ since $R\Gamma$ is free. That $H^{d-1}(\Gamma; R\Gamma)=0$ follows by recalling that $C$ injects into $C^{**}$ and letting $M=R\Gamma$ in the exact sequence
$$0 \rightarrow H^{d-1} (\Gamma ; M) \rightarrow C \otimes _{R\Gamma} M \rightarrow \Hom_{R\Gamma}(C^*,M)$$

\end{proof}

\section{Translation from topology} \label{sec:trans_to_topology}

\subsection{Cohomology compactly supported over each compact subcomplex.} If a
group $\Gamma$ has a finite Eilenberg-Maclane complex
$X = K(\Gamma, 1)$ with universal cover $\tilde X$ then for any ring
$R$ there is an isomorphism
$H^*(\Gamma; R\Gamma) = H_c^*(\tilde X; R)$. 
In this section we provide an alternative topological characterization
of $H^*(\Gamma, R\Gamma)$ in the case that $X$ is not finite. Our
proof uses standard techniques, which we include for completeness.

Suppose $X$ is a locally finite cell complex with an action by a group
$\Gamma$, and let $\pi: X \to \Gamma\backslash X$ denote the quotient map. Let
$C^*(X; R)$ denote the cellular cochain complex of $X$ with
coefficients in a ring $R$. Define a subcomplex
$\ccu{k}(X; R) \leq C^k(X; R)$ to contain cochains $\phi \in C^k(X; R)$ such
that for every $k$-cell $\sigma \in \Gamma\backslash X$ we have
$\phi(\tilde \sigma) = 0$ for all but finitely many
$\tilde \sigma \in \pi^{-1}(\sigma)$. Then
$d(\ccu{k}(X; R)) \subseteq \ccu{k+1}(X; R)$, and we let $\hcu{*}(X; R)$ be the
cohomology of this complex. We suppress the dependence on the action
of $\Gamma$ from the notation.

\begin{proposition} \label{prop:ourcohomology}
  Suppose $X$ is a locally finite, acyclic cell complex and $\Gamma$
  is a group acting on $X$ with cell stabilizers that are finite and
  preserve orientation. Then
  \[
  H^*(\Gamma; R\Gamma) \cong \hcu{*}(X; R)
  \]
\end{proposition}

\begin{proof}
  Recall that the
  equivariant cohomology of the pair $(X,\Gamma)$ with coefficients in
  $R\Gamma$ is defined as
  \[
  H_\Gamma^*(X; R\Gamma) = H^*(\Gamma; C^*(X; R\Gamma))
  \] 
  There is an isomorphism (cf. \cite[VII.7.3, p173]{Brown})
  \[
  H^*(\Gamma; R\Gamma) \cong H_\Gamma^*(X; R\Gamma)
  \]
  
  There is a spectral sequence (cf.\ \cite[p169]{Brown}) 
  \[
  E_1^{pq} = H^q(\Gamma; C^p(X; R\Gamma)) \implies H_\Gamma^{p+q}(X;
  R\Gamma). 
  \]
  We will show $H^q(\Gamma; C^p(X; R\Gamma)) = 0$ for all $q>0$. Let
  $X_p$ denote the set of $p$-cells in $X$ and let $\Sigma_p$ be a set
  of representatives for $\Gamma \backslash X_p$.  Letting
  $\Gamma_\sigma$ denote the stabilizer of $\sigma\in \Sigma_p$, there
  is a decomposition
  \begin{align*}
  C^p(X; R\Gamma) &= \Hom(C_p(X), R\Gamma) \\
                   &\cong \prod_{\sigma\in
                     X_p} R\Gamma \\
                   &\cong \prod_{\sigma \in \Sigma_p}
                     \Coind_{\Gamma_\sigma}^\Gamma( R\Gamma ).
  \end{align*}
  Therefore there is a decomposition of cohomology
  \begin{align*}
    H^q(\Gamma; C^p(X; R\Gamma))
    &\cong H^q \Bigl( \Gamma; \prod_{\sigma \in \Sigma_p}
      \Coind_{\Gamma_\sigma}^\Gamma( R\Gamma ) \Bigr)\\
    &\cong \prod_{\sigma \in \Sigma_p} H^q(\Gamma;
      \Coind_{\Gamma_\sigma}^\Gamma( R\Gamma )) 
  \end{align*}
  Applying Shapiro's lemma yields
  \[
  H^q(\Gamma; C^p(X; R\Gamma)) \cong \prod_{\sigma \in \Sigma_p}
  H^q(\Gamma_\sigma; R\Gamma)
  \]
  Because $\Gamma_\sigma$ is finite, there is an isomorphism of
  $R\Gamma_\sigma$-modules
  $R\Gamma \cong \Coind_{\{1\}}^{\Gamma_\sigma}( \oplus_{\Sigma_p}
  R)$. Therefore another use of Shapiro's lemma shows that
  $H^q(\Gamma_\sigma; R\Gamma) = 0$ for $q>0$. Recall that
  $H^0(\Gamma_\sigma; R\Gamma) \cong (R\Gamma)^{\Gamma_\sigma}$.

  It follows from the above that $H^*(\Gamma; R\Gamma)$ is the
  cohomology of the cochain complex
  \begin{equation} \label{simplespecseq}
    \prod_{\sigma\in \Sigma_0} (R\Gamma)^{\Gamma_\sigma} \to
    \prod_{\sigma\in \Sigma_1} (R\Gamma)^{\Gamma_\sigma} \to
    \prod_{\sigma\in \Sigma_2} (R\Gamma)^{\Gamma_\sigma} \to \dotsb
  \end{equation}
  We will show this chain complex is isomorphic to the chain complex $\{
  \ccu{k}(X, R) \}$. First we compute the coboundary maps of
  (\ref{simplespecseq}). The isomorphism 
  \[
  C^p(X; R\Gamma)^\Gamma \to \prod_{\sigma \in \Sigma_p}
  (R\Gamma)^{\Gamma_\sigma}
  \]
  sends a map $\phi: C_p(X; R) \to R\Gamma$ to the function
  $\xi: \Sigma_p \to R\Gamma$ defined by $\xi(\sigma) = \phi([\sigma])$
  for any $p$-cell $\sigma \in \Sigma_p$. From $\Gamma$-equivariance of
  $\phi$ computation shows that the coboundary operator on the complex
  (\ref{simplespecseq}) is given by
  \[
  d\xi(\sigma) = \sum_{i}r_i \gamma_i \xi(\sigma_i)
  \]
  if $\partial [\sigma] = \sum n_i \gamma_i [\sigma_i]$ for ring elements
  $r_i\in R$, group elements $\gamma_i \in \Gamma$ and simplices
  $\sigma_i\in \Sigma_p$.
  
  Define an isomorphism 
  \[
  \Theta : \prod_{\sigma\in \Sigma_p} (R\Gamma)^{\Gamma_\sigma} \to
  \ccu{p}(X, R)
  \]
  as follows: given $\phi\in \prod_{\sigma\in \Sigma_p} (R\Gamma)^{\Gamma_\sigma}$,
  for any $p$-simplex $\rho$ in $X$ choose $\sigma\in \Sigma_p$ and
  $\gamma\in \Gamma$ such that $\rho = \gamma \sigma$ and set
  \[
  \Theta\phi(\rho) = [\phi(\sigma)]_{\gamma^{-1}}
  \]
  Here $[x]_{\gamma^{-1}}$ denotes the coefficient of $[\gamma^{-1}]$ in
  the formal sum $x \in R\Gamma$. Note $\sigma$ is uniquely specified
  by $\rho$ and $\gamma$ is unique up to right multiplication by
  elements of $\Gamma_\sigma$. Because each $\phi(\sigma)$ is
  $\Gamma_\sigma$-invariant, $\Theta$ does not depend on choice of
  $\gamma$. Moreover, any two $p$-cells in $X$ that belong to the same
  $\Gamma$ orbit will correspond to the same cell $\sigma$ in the above
  equation. Since there are only finitely many terms in the formal sum
  $\phi(\sigma)$, the map $\Theta\phi$ is finitely supported above each
  cell in $X$. Therefore $\Theta$ determines a well-defined homomorphism
  of $\Gamma$-modules.
  
  It is clear that $\Theta$ is injective. To see that $\Theta$ is
  surjective, define an inverse $\Theta^{-1}$ by setting
  $[\Theta^{-1}\xi(\sigma)]_{\gamma} = \xi(\gamma^{-1}\sigma)$. It
  remains only to see that $\Theta$ is compatible with the coboundary
  maps. Suppose $\rho$ is a $(p+1)$-cell in $X$ and
  $\partial [\rho] = \sum_i r_i [\delta_i]$ for ring elements
  $r_i \in R$ and $p$-cells $\delta_i$. For each $i$, write
  $\delta_i = \gamma_i \sigma_i$ for $\gamma_i \in \Gamma$ and
  $\sigma_i\in \Sigma_p$. Then
  \begin{align*}
    d[\Theta\phi](\rho) &= \Theta\phi (\partial \rho)\\
                        &= \sum_{i} r_i [\Theta\phi](\delta_i)\\
                        &= \sum_i r_i [\phi(\sigma_i)]_{\gamma_i^{-1}}.
  \end{align*}
  On the other hand, note that if $\rho = \gamma \sigma$ for $\gamma\in
  \Gamma$ and $\sigma\in \Sigma_{p+1}$, then 
  \[\partial [\sigma] = \sum_i r_i (\gamma^{-1}\gamma_i) [\sigma_i]\]
  Therefore
  \begin{align*}
    \Theta(d\phi)(\rho) &= [(d\phi)(\sigma)]_{\gamma^{-1}}\\
                        &= [\phi(\partial\sigma)]_{\gamma^{-1}}\\
                        &= \left[ \sum_i r_i (\gamma^{-1} \gamma_i) \phi(\sigma_i)
                          \right]_{\gamma^{-1}}  \\
                        &= \sum_i r_i [(\gamma^{-1}
                          \gamma_i)\phi(\sigma_i)]_{\gamma^{-1}} \\
                        &= \sum_i r_i
                          [\phi(\sigma_i)]_{\gamma_i^{-1}}
  \end{align*}
  Thus $\Theta$ commutes with the coboundary operators and hence is an
  isomorphism of chain complexes. This completes the proof.
\end{proof}

\begin{lemma} \label{lem:commaction}
  Let $X$ and $\Gamma$ be as in Proposition
  \ref{prop:ourcohomology}. If $G$ is a locally compact group acting
  cellulary on $X$ and $\Gamma \leq G$ then $\Comm_G(\Gamma)$
  acts on $\hcu{*}(X; R)$.
\end{lemma}

\begin{proof}
  Given $\phi\in \ccu{k}(X)$ and $g \in \Comm_G(\Gamma)$, define
  $(g \phi) (\sigma) = \phi(g^{-1} \sigma)$. The condition that
  $g\phi \in \ccu{k}(X)$ is equivalent to the condition that
  $\supp(g\phi) \cap \Gamma K$ is compact for any compact set
  $K\subseteq X$, which is equivalent to
  $\supp(\phi)\cap g^{-1}\Gamma K$ being compact for any
  compact $K$. Fix a compact $K\subseteq X$. Because
  $g \in \Comm_G(\Gamma)$, there is some compact $K'\subseteq X$ such
  that $g\Gamma K \subseteq \Gamma K'$. Therefore $\supp(\phi) \cap
  g^{-1} \Gamma K \subseteq \supp(\phi)\cap \Gamma K'$. The latter is
  compact because $\phi \in \ccu{k}(X)$, so $g\phi \in
  \ccu{k}(X)$. This action commutes with coboundary maps, so induces
  an action on cohomology.
\end{proof}

\subsection{Computing $\hcu{*}(X)$.}

Let $X$ and $\Gamma$ be as in Proposition
\ref{prop:ourcohomology}. Suppose there are subcomplexes
$X_1\subseteq X_2 \subseteq X_3\subseteq \dotsb \subseteq X$ such that
each $X_k$ is closed and $\Gamma$-invariant, each quotient
$\Gamma \backslash X_k$ is compact, and $X = \bigcup_i X_i$. The
compactly supported cellular cochain complexes $C_c^*(X_n; R)$ form a
codirected system under the restriction maps $r_{i,j}^*$ induced by
inclusions $r_{i,j} : X_i \to X_j$ for $i < j$.

The chain complex $\ccu{k}(X; R)$ is the inverse limit of the system
of chain complexes $C_c^k(X_n; R)$ and each restriction map $r_{i,j}$
is surjective on the chain level. It follows (see for example the
``Variant'' following 
\cite[3.5.8, p84]{weibelbook}) that for any $k$ there is a short exact
sequence of cohomology
\begin{equation} \label{eq:limlim1}
0\to \varprojlim{}^1 H_c^{k-1}(X_n; R) \to \hcu{k}(X; R) \to \varprojlim
H_c^{k}(X_n; R) \to 0
\end{equation}
Recall that $\varprojlim^1 H_c^{k-1}(X_n; R)$ is the cokernel of the
map 
\[\Delta: \prod_{n=0}^\infty H_c^{k-1}(X_n; R) \to
\prod_{n=0}^\infty H_c^{k-1}(X_n; R)
\]
defined by 
\[
\Delta(x_0, x_1, x_2, \ldots) = (x_o - r_{0,1}(x_1), x_1-
r_{1,2}(x_2), \ldots)
\]

As a straightforward application of the short exact sequence
\eqref{eq:limlim1} we have:

\begin{proposition} \label{prop:hcu_ex}
  Suppose $X$ is a locally finite cell complex with an action of a
  group $\Gamma$. Suppose $X$ is the union of an increasing sequence
  of $\Gamma$-invariant subcomplexes $X_n$ each with cocompact
  $\Gamma$ action. If there is some integer $d$ such that
  $H_c^{*}(X_n; R)$ is concentrated in dimension $d$ for all $n$ then
  $\hcu{*}(X; R)$ is concentrated in dimensions $d$ and $d+1$. In
  particular, if $X$ is $d$-dimensional then
  \[
  \hcu{k}(X; R) = \begin{cases} \varprojlim_n H_c^d(X_n; R) & k=d \\
    0 & k\neq d. \end{cases}
  \]
\end{proposition}

 \section{Statement of reduction theory}\label{s:red}

In this section we'll review the necessary results needed from reduction theory for our proof of Theorem~\ref{t:mt}. The results in this section are not new, and can be derived from Behr \cite{Be} and Harder \cite{H}, although there are some minor differences between our treatment of reduction theory here and other versions already existing in the literature. A point of difference in the proof of our formulation of these results compared with formulations in other papers, is that we'll use the reduction theory from Bestvina-Eskin-Wortman \cite{B-E-W} as an input, which has the advantage, though not directly applied in this paper, of being equally applicable to arithmetic groups defined with respect to a number field. See also Bux-Wortman \cite{B-W2} and Bux-K\"ohl-Witzel \cite{B-K-W}.

 \subsection{Algebraic form of reduction theory}

In this section and the next we assume that $K$ is a global field with a ring of $S$-integers $\mathcal{O}_S \leq K$ and that $\mathbf{G}$ is a noncommutative, absolutely almost simple, $K$-isotropic, $K$-group with $\text{rank}_{K_v}({\bf G})=1$ for all $v \in S$.

Let $\bf P$ be  a proper $K$-parabolic subgroup of $\bf{G}$.  Let  $\bf A$ be a maximal $K$-split torus in $\bf P$.

From the root system for $({\bf G},{\bf A})$, we denote the simple root for the positive roots with respect to $\bf P$ by $\alpha _0$.

We let ${\bf Z_G}({\bf A})$ be the centralizer of ${\bf A}$ in $\bf G$ so that ${\bf Z_G}({\bf A})={\bf M}{\bf A}$ where ${\bf M}$ is a reductive $K$-group with $K$-anisotropic center. We let $\bf U$ be the unipotent radical of $\bf P$, so that ${\bf P}={\bf U}{\bf M}{\bf A}$. The Levi subgroup ${\bf M}{\bf A}$ normalizes the unipotent radical ${\bf U}$, and elements of ${\bf A}$ commute with those of ${\bf M}$.

We denote the product over $S$ of local points of a $K$-group by ``unbolding", so that, for example, $$G=\prod _{v\in S}{\bf G}(K_v)$$

 We let $\mathcal{P}$ be the set of proper $K$-parabolic subgroups of $\bf G$. If ${\bf Q} \in \mathcal{P}$, then $\bf Q$ is conjugate in ${\bf G}(K)$ to ${\bf P}$. We let $$\Lambda _{\bf Q} = \{\, \gamma f \in  {\bf G}(\mathcal{O}_S)F \mid \, ({\gamma f}){\bf P}({\gamma f})^{-1}={\bf Q} \,\}$$ where $F \se {\bf G}(K)$ is a finite set of coset representatives for  ${\bf G}(\mathcal{O_S}) \backslash {\bf G}(K) / {\bf P}(K)$. Note that if $\gamma_1 f_1, \gamma_2 f_2 \in \Lambda_{\bf Q}$, then $f_1=f_2$.

Given any $a=(a_v)_{v \in S} \in A$, we let 
$$|\alpha_0 (a)|=\prod_{v \in S} |\alpha _0 (a_v)|_v$$ where $|\cdot |_v$ is the $v$-adic norm on $K_v$.

Given any $t>0$, we let $$A^+(t)=\{\, a \in A \mid |\alpha _0 (a)| \geq t \,\}$$ 
and for $t >0$, we let $$R_{\bf Q}(t) =\Lambda _{\bf Q} U M A^+(t)$$

The following is a special case of Proposition 9 from Bestvina-Eskin-Wortman \cite{B-E-W}.

\begin{proposition}\label{p:pu}
There exists a bounded set $B _0 \se G$, and given any $N_0 \geq 0$, there exists $t_0 >1$ and a second bounded set $B_{1} \se G$  such that 
\begin{quote}
$(i)$ $G=\bigcup_{{\bf Q} \in \mathcal{P}}R_{\bf Q}(1)B_0$;
\medskip

\noindent $(ii)$ if ${\bf Q}, {\bf Q '} \in \mathcal{P} $ and ${\bf Q} \neq {\bf Q'}$, then  the distance between \newline \indent $R_{\bf Q}(t_0)B_0$  and $R_{\bf Q'}(t_0)B_0$ is at least $N_0$;
\medskip

\noindent $(iii)$ ${\bf G}(\mathcal{O}_S) \cap R_{\bf Q}(t_0)B_0 = \emptyset $; and
\medskip

\noindent $(iv)$ $G - \big( \bigcup_{{\bf Q} \in \mathcal{P}} R_{\bf Q}(2t_{0})B_{0} \big)$ is contained  \newline \indent in $  {\bf G}(\mathcal{O}_S)B_{ 1}$.

\medskip
\end{quote}
\end{proposition}

 \subsection{Geometric form of reduction theory}

We will now reformulate Proposition~\ref{p:pu} into a more explicit geometric statement in the form of Proposition~\ref{p:prune2} below.

For $v \in S$, we let $X_v$ be the Euclidean building for ${\bf G}(K_v)$, so that $X_v$ is a tree. We let $X_S=\prod _{v \in S} X_v$.

Let $\Sigma _v \se X_v$ be the geodesic that ${\bf A} (K_v)$ acts on by translations. We let $\Sigma _S= \prod_{v \in S} \Sigma_v$, so that $\Sigma _S $ is isometric to the Euclidean space $\mathbb{R}^{|S|}$. 

We define a linear functional $\widehat{\alpha _0}:\Sigma _S \rightarrow \mathbb{R}$ by associating a basepoint $e \in \Sigma _S$ with the origin as follows: $$\widehat{\alpha_0}(ae)=\log_{p}|\alpha_0(a)|$$ for $a \in  A$. The action of $A$ on $e$ factors through $\mathbb{Z}^{|S|}$, where $\widehat{\alpha_0}$ is linear, so $\widehat{\alpha_0}$ extends to a functional on all of $\Sigma_S$. Furthermore, $\widehat{\alpha_0}$ is nonzero since there is some $a$ with $|\alpha_0 (a)|\neq 1$.

For any $r \in \mathbb{R}$, we let $\Sigma _{S,r} \se \Sigma _S$ be $$\Sigma _{S,r} = \{\, x \in \Sigma _S \mid \widehat{\alpha_0}(x)=r \, \}$$
Thus, $\Sigma _{S,r} $ is a hyperplane in $ \Sigma _S$ that is a finite Hausdorff distance from  ${\bf  A}(\mathcal{O}_S) e \se \Sigma _{S,0}$.

Note that $\Sigma _{S,r}$ is not singular if $|S|>1$. That is to say, the projection of $\Sigma _{S,r}$ to each $\Sigma _v$ is surjective if $|S|>1$. Indeed, to verify this claim observe that if $v \in S$, then ${\bf  A}(\mathcal{O}_S) $ has dense projection to    ${\bf A} (K_v)$, and thus acts cocompactly on $\Sigma _v$.

Now consider the geodesics $\Sigma _v$ to be parameterized as unit speed $\Sigma _v : \mathbb{R} \rightarrow X_v$ with $\Sigma_v(\infty)={\bf P}$.
From our description of $\widehat{\alpha_0}: \Sigma _S \rightarrow \mathbb{R}$, we see that there are positive real numbers $\lambda _v$ such that if $\rho_S :\mathbb{R} \rightarrow X_S$ is given by $\rho_S(t)=(\Sigma _v (\lambda _v t))_{v\in S}$, and if $\beta _{\rho_S} : X_S \rightarrow \mathbb{R}$ is the Busemann function for $\rho _S$ -- that is if  $x \in X_S$, and $d$ is the distance function on $X_S$, then $$\beta _{\rho_S} (x)=\lim_{t \to \infty}(t - d(x,\rho_S(t)))$$
-- then  $\beta _{\rho_S} $ restricted to $\Sigma _S$ is exactly  $\widehat{\alpha_0}$.

Let 
\begin{align*} 
\Sigma ^+_{S,r} & = \{\, x \in \Sigma _S \mid \widehat{\alpha_0}(x)\geq r \, \} \\
& = \{\, x \in \Sigma _S \mid \beta _{\rho_S}(x)\geq r \, \} 
\end{align*} 
so that $\Sigma ^+_{S,r}$ is a half space in $\Sigma _S$ whose boundary equals $\Sigma _{S,r}$.

We let $$B_{\mathbf{P},S,r}= \{\, x \in X _S \mid \beta _{\rho_S}(x)\geq r \, \}$$ and
$$Y_{\mathbf{P},S,r}= \{\, x \in X _S \mid \beta _{\rho_S}(x) = r \, \}$$

\begin{lemma}
$B_{\mathbf{P},S,r}=UM\Sigma _{S,r}^+$ and $Y_{\mathbf{P},S,r}=UM\Sigma _{S,r}$.
\end{lemma}

\begin{proof}
$\bf M$ is contained in both $\bf P$ and the parabolic group opposite to $\bf P$ with respect to $\bf A$. 
Also note that ${\bf M} (K_v)$ is compact for all $v \in S$. It follows that ${\bf M} (K_v)$ fixes $\Sigma _v$ pointwise, and thus that $M$ fixes $\Sigma_S$ pointwise.
Therefore,  $UM\Sigma _{S,r}^+=U\Sigma _{S,r}^+$.

Elements of ${\bf U}(K_v)$ fix unbounded positive rays in $\Sigma _v$, thus elements of $U$ fix pointwise a subray of $\rho _S$, thus $\beta _{\rho_S} $ is invariant under multiplication by $U$. Therefore $UB_{\mathbf{P},S,r}=B_{\mathbf{P},S,r}$, so $UM\Sigma _{S,r}^+\se B_{\mathbf{P},S,r}$ follows from $\Sigma _{S,r}^+\se B_{\mathbf{P},S,r}$.

To see that $ B_{\mathbf{P},S,r} \se U\Sigma _{S,r}^+$, let $x \in  B_{\mathbf{P},S,r}$. Since $X_v = {\bf U}(K_v) \Sigma _v$, we see that $x=u(x_v)_{v \in S}$ for some $u \in U$ and $x_v \in \Sigma_v$. Thus, $x \in U\Sigma _{S,r}^+$, again, since $\beta _{\rho_S} $ is invariant under multiplication by $U$.

That $Y_{\mathbf{P},S,r}=UM\Sigma _{S,r}$ follows similarly.
\end{proof}

Given $t \in \mathbb{R}$, let $r_t \in \mathbb{R}$ be the supremum of all $r \in \mathbb{R}$ such that  $\Sigma _{S,r}^+$ contains $ A^+(t)e$. Notice that there is some $C>0$, independent of $t$, such that the Hausdorff distance between $ A^+(t)e$ and $\Sigma _{S,r_t}^+$ is bounded by $C$. Notice also that $t \mapsto r_t$ is an increasing function.

\begin{lemma}
The Hausdorff distance between $ UMA^+(t)e$ and $B_{\mathbf{P},S,r_t}$ is bounded independent of $t$.\end{lemma}

\begin{proof}
Because the Hausdorff distance between $ A^+(t)e$ and $\Sigma _{S,r_t}^+$ is bounded, the Hausdorff distance between $ UMA^+(t)e$ and $UM\Sigma _{S,r_t}^+=B_{\mathbf{P},S,r}$ is bounded.
\end{proof}

\begin{lemma} \label{l:horoball-action}
Let $\bf Q \in \mathcal{P}$. If $\gamma \in {\bf G}(\mathcal{O}_S)$ and $f \in F$ are such that $\gamma f \in \Lambda _{\bf Q}$, then for any $r$, we have ${\bf Q}(\mathcal{O}_S) \gamma f B_{\mathbf{P},S,r}=\gamma f B_{\mathbf{P},S,r}$.
\end{lemma}

\begin{proof}
Note that as  $B_{\mathbf{P},S,r}$ is given by the Busemann function for $\rho _S$, $\gamma f  B_{\mathbf{P},S,r}$ is given by the Busemann function for $\gamma f \rho_S$. 

Since $\gamma f \in \Lambda _{\bf Q}$, the positive end of each $\gamma f \Sigma _v $ limits to $\bf Q$. Thus, if $g \in {\bf Q}(K_v)$, then $\gamma f \Sigma _v $ and $g \gamma f \Sigma _v$ intersect in a positive ray. Hence, if $g \in Q$, then $g \gamma f \rho _S$ is a finite Hausdorff distance from $\gamma f \rho _S$ and $g \gamma f B_{\mathbf{P},S,r}=\gamma f B_{\mathbf{P},S,r_g}$ for some $r_g \in \mathbb{R}$. By replacing $g$ with its inverse, we may assume that $r_g\geq r$.

We may assume that the set $B_0$ from Proposition~\ref{p:pu} is a sufficiently large neighborhood of $1 \in G$, independent of $g$, so that, in particular there is a set $B' \se B_0$ containing the point stabilizer of $1$ and such that $B'B'\se B_0$, and by the previous lemma, such that $UMA^+(t)B'e$ contains every vertex of $B_{\mathbf{P},S,r_t}$.

Let $t_0$ be as in Proposition~\ref{p:pu}. If $r_g \neq r$,  then for sufficiently large $n$ we have $g^n e \in \gamma f B_{\mathbf{P},S,r_{t_0}}$. Hence,  $g^n e \in \gamma f UMA^+(t_0) B' e$. Therefore, $g^n \in \gamma f UMA^+(t_0) B' B' \se R_{\bf Q}(t_0)B_0$. We conclude, by Proposition~\ref{p:pu} part (iii), that $g \notin {\bf G}(\mathcal{O}_S)$.
\end{proof}

If $\bf Q \in \mathcal{P}$, we define $$ B_{\mathbf{Q},S,r}=\gamma f B_{\mathbf{P},S,r}$$ for any $\gamma f \in \Lambda _Q$. This is well-defined by the previous lemma, and we also see that ${\bf Q}(\mathcal{O}_S) B_{\mathbf{Q},S,r} = B_{\mathbf{Q},S,r}$ and that the Hausdorff distance between $R_{\bf Q}(t)B_0e$ and $B_{\mathbf{Q},S,r_t}$ is bounded independent of $t$ or $\bf Q$. Using this and that the orbit map $G \rightarrow Ge \se X_S$ is proper, we deduce from Proposition~\ref{p:pu} the following

\begin{proposition}\label{p:prune2}
There exists some $r _0 \in \mathbb{R}$, and given any $N \geq 0$, there is some $ r_1 >r_0$ such that
\begin{quote}
$(i)$ $\bigcup_{{\bf Q} \in \mathcal{P}}B_{\mathbf{Q},S,r_0} = X_S$;
\medskip

\noindent $(ii)$ if ${\bf Q}, {\bf Q '} \in \mathcal{P} $ and ${\bf Q} \neq {\bf Q'}$, then  the distance between $B_{\mathbf{Q},S,r_1}$  and $B_{\mathbf{Q}',S,r_1}$ is at least $N$;
\medskip

\noindent $(iii)$ ${\bf G}(\mathcal{O}_S)e \cap B_{\mathbf{Q},S,r_1} = \emptyset $; and
\medskip

\noindent $(iv)$ $X_S - \big( \bigcup_{{\bf Q} \in \mathcal{P}} B_{\mathbf{Q},S,r_1} \big)$ is a finite Hausdorff distance from $  {\bf G}(\mathcal{O}_S)e$.
\end{quote}
\end{proposition}

For any $r \in \mathbb{R}$, we let $X_{S,r}$ be the closure in $X_S$ of $X_S-(\cup _{{\bf Q} \in \mathcal{P}} B_{\mathbf{Q},S,r})$. 

\begin{lemma} \label{l:prop-cocpt}
  For $r \gg 0$, ${\bf G}(\mathcal{O}_S)$ acts properly and cocompactly on $X_{S,r}$.
\end{lemma}

\begin{proof}
Let $\gamma \in {\bf G}(\mathcal{O}_S)$. Then $\gamma B_{\mathbf{Q},S,r}=B_{\gamma \mathbf{Q}\gamma ^{-1},S,r}$ so ${\bf G}(\mathcal{O}_S)$ acts on $\cup _{{\bf Q} \in \mathcal{P}} B_{\mathbf{Q},S,r}$ and thus on $X_{S,r}$.

Since ${\bf G}(\mathcal{O}_S)$ acts properly on $X_{S}$, it acts properly on $X_{S,r}$.

That  ${\bf G}(\mathcal{O}_S)$ acts cocompactly on $X_{S,r}$ follows from (iv) of Proposition~\ref{p:prune2}.
\end{proof}

\section{Cohomology of the complement of disjoint horoballs}
\label{sec:horoball_comp}

In this section, we'll examine the cohomology of subspaces of $X_S$ that include spaces of the form $X_{S,r}$, but are slightly more general in that we will allow ourselves to set the height of each horoball individually, rather than use a single parameter to define the height of all horoballs simultaneously. Precisely, for any tuple  $(r_{\bf Q})_{{\bf Q} \in \mathcal{P}} \in ( \mathbb{R} \cup \{ \infty\} )^{\mathcal{P}}$, we let $X_{S,(r_{\bf Q})}$ be the closure of $X_S-(\cup _{{\bf Q} \in \mathcal{P}} B_{\mathbf{Q},S,r_{\bf Q}})$ in $X_S$, where $B_{\mathbf{Q},S,\infty}$ is taken to be the empty set. 

We shall call a tuple $(r_{\bf Q})_{{\bf Q} \in \mathcal{P}} \in ( \mathbb{R} \cup \{ \infty\} )^{\mathcal{P}}$ \emph{sufficiently large} if the resulting sets $B_{\mathbf{Q},S,r_{\bf Q}}$ are pairwise disjoint, and if their pairwise distance is bounded below by a constant that is sufficiently large. It's known that if $(r_{\bf Q})_{{\bf Q} \in \mathcal{P}}$ is sufficiently large then $X_{S,(r_{\bf Q})}$ is $(|S|-2)$-connected but not $(|S|-1)$-connected (see Stuhler \cite{St}, Bux-Wortman \cite{B-W2}, and Bux-K\"ohl-Witzel \cite{B-K-W}), but these topological properties are not directly relevant to this paper. What we require in this paper, and what we will prove in this section, is that $H_c^{k}(X_{S,(r_{\bf Q})})=0$ if $k \leq |S|-1$ and $(r_{\bf Q})_{{\bf Q} \in \mathcal{P}}$ is sufficiently large. (See Proposition~\ref{p:397} below.) We will begin an inductive proof of this claim by observing that the claim is true when $|S|=1$.

\begin{lemma}
If $|S|=1$, and if $(r_{\bf Q})_{{\bf Q} \in \mathcal{P}}$ is sufficiently large, then the group $H_c^{0}(X_{S,(r_{\bf Q})})$ is trivial, where coefficients are in a ring $R$.
\end{lemma}

\begin{proof}
In this case, $X_S$ is a tree, and we want to show that the components of $X_{S,(r_{\bf Q})}$ are unbounded. Indeed, choose an edge $e_0 \in X_{S,(r_{\bf Q})}$. Because $(r_{\bf Q})_{\bf Q}$ is sufficiently large, there is an adjacent edge $e_1 \in X_{S,(r_{\bf Q})}$, and we can continue in this fashion to create an path of infinite length in $X_{S,(r_{\bf Q})}$ that begins with $e_0$.
\end{proof}

Our proof of Proposition~\ref{p:397} will include an investigation of spaces that are quite similar to the space $X_{S,(r_{\bf Q})}$. Precisely, 
for any $(r_{\bf Q})_{{\bf Q} \in \mathcal{P}}$, let $W_{S,(r_{\bf Q})}$ be the subcomplex of $X_S$ consisting of all cells of $X_S$ that are contained in $X_{S,(r_{\bf Q})}$. To see that there isn't much difference between $X_{S,(r_{\bf Q})}$ and $W_{S,(r_{\bf Q})}$ we have

\begin{lemma} If the tuple $(r_{\bf Q})_{{\bf Q} \in \mathcal{P}}$ is sufficiently large, then there is a proper homotopy equivalence between $W_{S,(r_{\bf Q})}$ and $X_{S,(r_{\bf Q})}$.
\end{lemma}

\begin{proof} The proof is an observation through Morse theory. Suppose that $\mathfrak{C} \se X_S$ is a chamber that intersects $X_{S,(r_{\bf Q})}$ nontrivially, but is not contained in $X_{S,(r_{\bf Q})}$, and thus is not contained in $W_{S,(r_{\bf Q})}$.

Because $(r_{\bf Q})_{{\bf Q} \in \mathcal{P}}$ is sufficiently large,  $\mathfrak{C}$ intersects $B_{\mathbf{Q},S,r_{\bf Q}}$ for a unique $\bf Q$. Recall that $B_{\mathbf{Q},S,r_{\bf Q}}$ is defined as the inverse image of a positive ray with respect to  the Busemann function $\beta _{\gamma f \rho _S} : X_S \rightarrow \mathbb{R}$ associated to the geodesic ${\gamma f \rho _S} \se X_S$ where $\gamma f \in \Lambda_{\bf Q}$.

Let $(x_v)_{v\in S}$ be the maximum point of $\mathfrak{C}$ with respect to $\beta _{\gamma f \rho _S}$. Let $\mathcal{L}_v$ be the descending link of $x_v$ in the tree $X_v$ with respect to $\beta _{\gamma f \Sigma _v}:X_v \rightarrow \mathbb{R}$. We let $\mathcal{C}_v$ be the cone on $\mathcal{L}_v$ taken at $x_v$ in the tree $X_v$.

For $T \se S$, we let $K_T=\prod_{v \in T} \mathcal{C}_v \times \prod_{v \notin T} \mathcal{L}_v$.

Now we are assuming that $(x_v)_{v\in S} \notin X_{S,(r_{\bf Q})}$, and note that $K_S -  (x_v)_{v\in S}$ deformation retracts onto $ \cup _{v \in S} K_{S-v}$ in such a way that the homotopy is nonincreasing with respect to $\beta _{\gamma f \rho _S}$. Note further that the maximum points in any $K_{S-v_0}$ with respect to $\beta _{\gamma f \rho _S}$ are points of the form  $(y_v)_{v\in S} $ where $y_v=x_v$ if $v \neq v_0$, and if these points are not in $X_{S,(r_{\bf Q})}$, then we can further retract $K_{S-v_0}$ minus these maximums onto $\cup _{v \in S - v_0}K_{S-\{v_0,v\}}$. We continue in this fashion until all of $K_S$ has been retracted onto some union of $K_{T}$ with $K_T \subseteq X_{S,(r_{\bf Q})}$.

\end{proof}

In particular, the previous two lemmas show that 

\begin{lemma}\label{l:947hdbk3129p}
If $|S|=1$, and if $(r_{\bf Q})_{{\bf Q} \in \mathcal{P}}$ is sufficiently large, then $H_c^{0}(W_{S,(r_{\bf Q})})=0$.
\end{lemma}

This lemma will serve as the base step for our inductive proof that $H_c^{k}(W_{S,(r_{\bf Q})})=0$ if $ k\leq |S|-1$ and $(r_{\bf Q})_{{\bf Q} \in \mathcal{P}}$ is sufficiently large, which implies that  $H_c^{k}(X_{S,(r_{\bf Q})})=0$ if $ k\leq |S|-1$ and $(r_{\bf Q})_{{\bf Q} \in \mathcal{P}}$ is sufficiently large.

\subsection{Proper products}

Now we will focus on the case when $|S|\geq 2$. We choose some $w \in S$ and let $\pi _w : X_S \rightarrow X_{w}$ be the projection.

Note that by definition of $W_{S,(r_{\bf Q})}$, if $e$ is an edge in $X_w$, and if $e ^\circ$ is the interior of $e$, then $\pi_w |_{W_{S,(r_{\bf Q})}} : W_{S,(r_{\bf Q})} \rightarrow X_w$ has $\pi _w ^{-1} (e^\circ)=e^\circ \times Z_e$ for some complex $Z_e \se X_{S-w}$. Our inductive proof in the remainder of this section is aided by observing that the fibers $\pi_w$ restricted to one of these ``$W$ spaces'' is another ``$W$ space".

\begin{lemma}\label{l:fib}
For any edge $e \se X_w$, $Z_e = W_{S-w,(s^e_{\bf Q})}$ for some tuple $(s^e_{\bf Q})_{{\bf Q} \in \mathcal{P}}$.
Furthermore, by choosing $(r_{\bf Q})_{{\bf Q} \in \mathcal{P}}$ sufficiently large we may assume that $(s^e_{\bf Q})_{{\bf Q} \in \mathcal{P}}$ is sufficiently large for each edge $e \se X_w$.
\end{lemma}

\begin{proof}
Let $x_w \in X_w$ be the  endpoint of $e$ that maximizes $\beta _{\gamma f \Sigma _w}$ for $\gamma f \in \Lambda _{\bf Q}$. Then  a cell $\mathfrak{F}\se X_{S-w}$ is contained in $Z_e$ exactly if $$\beta _{\gamma f \rho _S}(e \times \mathfrak{F})\leq r_{\bf Q}$$ which is equivalent to 
$$ \beta _{\gamma f \rho _S}(x_w \times \mathfrak{F})\leq r_{\bf Q}$$ and thus to 
$\beta _{\gamma f \rho _{(S-w)}}( \mathfrak{F})\leq s^e_{\bf Q}$ for some $s^e_{\bf Q}$ depending
$\beta _{\gamma f \Sigma _w}(x_w)$, and thus on $e$. 
\end{proof}

\begin{lemma}
Let $\gamma f \in \Lambda _{\bf Q}$. If $e_1,e_2 \in X_w$ are edges, and if the maximum of $\beta _{\gamma f \Sigma _w}(e_1)$ is greater than or equal to the maximum of $\beta _{\gamma f \Sigma _w}(e_2)$, then $s^{e_1}_{\bf Q}\leq  s^{e_2}_{\bf Q}$. If $\beta _{\gamma f \Sigma _w}(e_1)=\beta _{\gamma f \Sigma _w}(e_2)$, then $s^{e_1}_{\bf Q} =  s^{e_2}_{\bf Q}$.
\end{lemma}

\begin{proof}
Let $\chi_w \se X_w$ be a geodesic that limits to $\mathbf{Q}$, and suppose that $e_1 \se \chi _w$.

First assume that that $e_2 \se \chi _w$. Then since $\beta _{\gamma f \Sigma _w}(e_2) \leq \beta _{\gamma f \Sigma _w}(e_1)$ we see that $s^{e_2}_{\bf Q}\geq  s^{e_1}_{\bf Q}$ as desired.

If $e_2$ is not contained in $\chi _w$, then there is some $u \in {\bf U}(K_w)$ such that $u\chi _w$ does contain $e_2$. The result follows from the above as $\beta _{\gamma f \rho _{S-w}}$ and $\beta _{\gamma f \Sigma _{w}}$ are invariant by translations of $ {\bf U}(K_w)$.
\end{proof}

Given a vertex $y \in X_w$, we let $E_y$ be the set of edges in $X_w$ that contain $y$. Then the previous lemma produces

\begin{lemma}\label{l:top}
For any vertex $y \in X_w$, and any parabolic ${\bf Q} \in {\mathcal{P}}$, either $\{s^e_{\bf Q}\}_{e \in E_y}$ contains a single value, or else $\{s^e_{\bf Q}\}_{e \in E_y}$ contains exactly two values, and the minimum value is realized by a unique edge in $E_y$.
\end{lemma}

\begin{proof}
For $\gamma f \in \Lambda _{\bf Q}$, observe that there is a unique edge containing $y$ that maximizes the Busemann function $\beta _{\gamma f \Sigma _w}$, and that the remaining edges minimize $\beta _{\gamma f \Sigma _w}$.
\end{proof}

In what follows, we'll denote the unique edge in $E_y$ from the proof of the previous lemma as $e(y,{\bf Q})$. Thus if $e , \epsilon \in E_y$, then $s^e_{\bf Q} \leq s^\epsilon_{\bf Q}$ if $e=e(y,{\bf Q})$, and $s^e_{\bf Q} = s^\epsilon_{\bf Q}$ if $e, \epsilon \neq e({y,{\bf Q}})$.

We will need one more related observation about the fibers of $\pi _w$ in the form of the following

\begin{lemma}\label{l:alt}
If there is a vertex $y \in X_w$, and a cell $\mathfrak{F} \se X_{S-w}$ such that $y \times \mathfrak{F} \se W_{S,(r_{\bf Q})}$, then $e \times \mathfrak{F} \se W_{S,(r_{\bf Q})}$
for each $e \in E_y -e(y,{\bf Q})$.
\end{lemma}

\begin{proof}
Let $\gamma f \in \Lambda _{\bf Q}$. Since $y \times \mathfrak{F} \se W_{S,(r_{\bf Q})}$, the values of $\beta _{\gamma f \rho _S}(y \times \mathfrak{F})$ are bounded above by $r_{\bf Q}$. Since $y$ maximizes the values of $e$ under $\beta _{\gamma f \Sigma _w}$, the values of 
$\beta _{\gamma f \rho _S}(e \times \mathfrak{F})$ are bounded above by $r_{\bf Q}$ as well.  That is, $e \times \mathfrak{F} \se W_{S,(r_{\bf Q})}$.
\end{proof}

\subsection{Cover by fibers}

Having collected some information about the fibers of $\pi_w |_{W_{S,(r_{\bf Q})}}$, we will now use a collection of fibers to create a cover for $W_{S,(r_{\bf Q})}$.

For any edge $e \se X_w$, let $F_e=e \times  W_{S-w,(s^e_{\bf Q})}$ where $ W_{S-w,(s^e_{\bf Q})}$ is as in Lemma~\ref{l:fib}.

\begin{lemma}\label{l:398268}
The collection $\{F_e\}$ taken over all edges $e \se X_w$ is a cover for $W_{S,(r_{\bf Q})}$.\end{lemma}

\begin{proof}
Suppose $\sigma \times \mathfrak{F}$ is a cell in $W_{S,(r_{\bf Q})}$, where $\sigma$ is a cell in an edge $e \se X_w$ and $\mathfrak{F}$ is a cell in $ X_{S-w}$.

If $\sigma =e$, then $\sigma \times \mathfrak{F} \se F_e$ by Lemma~\ref{l:fib}. If $\sigma$ is a vertex of $e$, say $y$, then by Lemma~\ref{l:alt}, there is some $e'$ such that $y \times \mathfrak{F} \se e' \times \mathfrak{F}\se F_{e'}$.
\end{proof}

 For any vertex $y \in X_w$, let $F_y=\cup_{e \in E_y}F_e$. 
Note that there is a proper homotopy equivalence between $F_y$ and $$\cup_ {y \in e} W_{S-w,(s^e_{\bf Q})}=W_{S-w,(\max_{e \in E_y}\{s^e_{\bf Q}\})}$$ given by retracting the star of $y$ in $X_w$ to the point $y$. 

 Further, if $e \in E_y$, then the inclusion $F_{e} \rightarrow F_y$, after proper homotopy equivalence, is the inclusion $ W_{S-w,(s^{e}_{\bf Q})} \rightarrow W_{S-w,(\max_{\epsilon \in E_y}\{s^\epsilon_{\bf Q}\})}$. In particular, if $e \in E_y$, the we can, and we shall, identify the map induced by inclusion $$\rho _{y,e}:H^{|S|-1}_c(F_y) \rightarrow H^{|S|-1}_c(F_e)$$ with the map $$\rho _{y,e}:H^{|S|-1}_c(W_{S-w,(\max_{\epsilon \in E_y}\{s^\epsilon_{\bf Q}\})}) \rightarrow H^{|S|-1}_c( W_{S-w,(s^{e}_{\bf Q})} )$$

\subsection{Maps between the cohomology of the fibers}

For  an edge $e \se X_w$, and a parabolic group ${\bf R} \in \mathcal{P}$, we let $\mathcal{S}_{e,{\bf R}} \se W_{S-w,(s^{e}_{\bf Q})} $ be the complex comprised of all cells $\mathcal{F} \se W_{S-w,(s^{e}_{\bf Q})} $ such that there is a cell $\mathcal{G} \se X_{S-w}$ containing $\mathcal{F}$ with $\beta _{\gamma f \rho_{(S-w)}}(\mathcal{G}) \nleq s^{e}_{\bf R}$ where $\gamma f \in \Lambda _{\bf R}$. 
Thus we may informally think of the boundary of 
$W_{S-w,(s^{e}_{\bf Q})} $ as $\amalg _{{\bf Q} \in \mathcal{P}} \mathcal{S}_{e,{\bf Q}}$.

Let $y \in X_w$ be a vertex, $e \in E_y$, and ${\bf R} \in {\mathcal{P}}$. We define $\mathcal{J}_{y,e,{\bf R}}$ to be the union of cells ${\mathcal{F}} \se  W_{S-w,(\max_{\epsilon \in E_y}\{s^\epsilon_{\bf Q}\})}$ such that  the maximum value of $\beta _{\gamma f \rho _{(S-w)}}(\mathcal{F})$ is greater than $s^e_{\bf R}$ for $\gamma f \in \Lambda _{\bf R}$. Notice that if  $\mathcal{J}_{y,e,{\bf R}}\neq \emptyset$, then $s^e_{\bf R}<s^\epsilon_{\bf R}$ for some $\epsilon \in E_y$, which, by Lemma~\ref{l:top}, implies that $e=e(y,{\bf R})$.

\begin{lemma}\label{l:Jtriv}
If $y \in X_w$, $e \in E_y$, and ${\bf R} \in {\mathcal{P}}$, then $H^{|S|-1}_c(\mathcal{J}_{y,e,{\bf R}})=0$.
\end{lemma}

\begin{proof}
We may assume that $\mathcal{J}_{y,e,{\bf R}}\neq \emptyset$, so that $s^e_{\bf R}<s^\epsilon_{\bf R}$ for $\epsilon \in E_y - e$. Then $\mathcal{J}_{y,e,{\bf R}}$ is the complex of cells ${\mathcal{F}}$ such that the maximum value of $\beta _{\gamma f \rho _{(S-w)}}(\mathcal{F})$ is greater than $s^e_{\bf R}$ but bounded above by $s^\epsilon_{\bf R}$.

Let ${\mathcal{F}}$ be a cell as in the above paragraph of dimension $|S|-1$ and assume that $\beta _{\gamma f \rho _{(S-w)}}(\mathcal{F})$ attains the minimal value for all such $\mathcal{F}$. Then we can retract $\mathcal{F}$ into $\partial \mathcal{F}$ along the direction of the geodesic $\gamma f \rho _{(S-w)}$. Repeat this process until $\mathcal{J}_{y,e,{\bf R}}$ is retracted onto a complex of dimension $|S|-2$.
\end{proof}

We let $K_{y,e} \se \mathcal{P}$ be the set of all ${\bf R} \in {\mathcal{P}}$ such that $\mathcal{J}_{y,e,{\bf R}}\neq \emptyset$. If $e,\epsilon \in E_y$, and  if $\mathcal{J}_{y,e,{\bf R}} $ and $\mathcal{J}_{y,\epsilon,{\bf R}}$ are each nonempty, then  $e=e(y,{\bf R})=\epsilon$. Therefore, if $e$ and $\epsilon$ are distinct, we have $K_{y,e}\cap K_{y,\epsilon}=\emptyset$ so that if we let $K_y=\cup _{e \in E_y}K_{y,e} \se \mathcal{P}$, then $$K_y=\amalg _{e \in E_y}K_{y,e}$$

\begin{lemma}\label{l:mvs}
Given a vertex $y \in X_w$ and $e \in E_y$,  $$W_{S-w,(\max_{\epsilon \in E_y}\{s^\epsilon_{\bf Q}\})}=W_{S-w,(s^{e}_{\bf Q})} \cup \Big( \amalg _{{\bf R} \in K_{y,e}} \mathcal{J}_{y,e(y,{\bf R}),{\bf R}}\Big)$$ Furthermore $$W_{S-w,(s^{e}_{\bf Q})} \cap \Big( \amalg _{{\bf R} \in K_{y,e}} \mathcal{J}_{y,e(y,{\bf R}),{\bf R}}\Big) = \amalg _{{\bf R} \in K_{y,e}} \mathcal{S}_{e(y,{\bf R}),{\bf R}}$$
\end{lemma}

\begin{proof}
By definition, for all ${\bf R} \in \mathcal{P}$, we have that $\mathcal{J}_{y,e,{\bf R}} $ and  $ W_{S-w,(s^{e}_{\bf Q})}$ are contained in $  W_{S-w,(\max_{\epsilon \in E_y}\{s^\epsilon_{\bf Q}\})}$.

If $\mathcal{F} \se   W_{S-w,(\max_{\epsilon \in E_y}\{s^\epsilon_{\bf Q}\})}$ is a cell, and if $\mathcal{F}$ is not contained in $W_{S-w,(s^{e}_{\bf Q})}$, then the maximum value of $\beta _{\gamma f \rho _{(S-w)}}(\mathcal{F})$ is greater than $s^e_{\bf R}$ for some ${\bf R}\in \mathcal{P}$ and $\gamma f \in \Lambda _{\bf R}$, so that $\mathcal{F} \se \mathcal{J}_{y,e,{\bf R}} $ which is to say that 
$$W_{S-w,(\max_{\epsilon \in E_y}\{s^\epsilon_{\bf Q}\})} \se W_{S-w,(s^{e}_{\bf Q})} \cup \bigcup_{{\bf R}\in \mathcal{P}}  \mathcal{J}_{y,e,{\bf R}}$$
 so we have equality. Furthermore, by the definition of $K_{y,e}$, and since $(s^{e}_{\bf Q})_{{\bf Q }\in \mathcal{P}}$ is sufficiently large, we have
$$W_{S-w,(\max_{\epsilon \in E_y}\{s^\epsilon_{\bf Q}\})} = W_{S-w,(s^{e}_{\bf Q})} \cup \amalg_{{\bf R}\in K_{y,e}}  \mathcal{J}_{y,e,{\bf R}}$$

Now suppose that there is a cell $\mathcal{F}$ contained in both $W_{S-w,(s^{e}_{\bf Q})}$ and $  \mathcal{J}_{y,e,{\bf R}}$ for some ${\bf R}\in K_{y,e}$. The latter inclusion implies that there is some $\mathcal{G} \se X_{S-w}$ such that the maximum value of $\beta _{\gamma f \rho _{(S-w)}}(\mathcal{G})$ is greater than $s^e_{\bf R}$ for $\gamma f \in \Lambda _{\bf R}$. That is, $\mathcal{F} \se \mathcal{S}_{e,{\bf R}}$.

To show the other inclusion, let $\mathcal{F} \se W_{S-w,(s^{e}_{\bf Q})} $ be such that there is a cell $\mathcal{G} \se X_{S-w}$ containing $\mathcal{F}$ with $\beta _{\gamma f \rho_{(S-w)}}(\mathcal{G}) \nleq s^{e}_{\bf R}$ for some ${\bf R} \in K_{y,e}$ where $\gamma f \in \Lambda _{\bf R}$. Then $\mathcal{F}\se   \mathcal{J}_{y,e,{\bf R}}$.
\end{proof}

We also have the following lemma whose proof is similar.

\begin{lemma}\label{l:mvs2}
Given a vertex $y \in X_w$,  $$W_{S-w,(\max_{\epsilon \in E_y}\{s^\epsilon_{\bf Q}\})}=W_{S-w,(s^{e(y,{\bf Q})}_{\bf Q})} \cup \Big( \amalg _{{\bf R} \in K_{y}} \mathcal{J}_{y,e(y,{\bf R}),{\bf R}}\Big)$$ Furthermore $$W_{S-w,(s^{e(y,{\bf Q})}_{\bf Q})} \cap \Big( \amalg _{{\bf R} \in K_{y}} \mathcal{J}_{y,e(y,{\bf R}),{\bf R}}\Big) = \amalg _{{\bf R} \in K_{y}} \mathcal{S}_{e(y,{\bf R}),{\bf R}}$$
\end{lemma}

 The Mayer-Vietoris sequence for the pair in Lemma~\ref{l:mvs} yields the  coboundary homomorphism 
$$\delta _{y,e} : \oplus _{{\bf R}\in K_{y,e}} H_c^{|S|-2}( \mathcal{S}_{e(y,{\bf R}),{\bf R}}) \rightarrow H_c^{|S|-1}(W_{S-w,(\max_{\epsilon \in E_y}\{s^\epsilon_{\bf Q}\})})$$

Similarly, Lemma~\ref{l:mvs2} yields
$$\delta _{y} : \oplus _{{\bf R} \in K_{y}} H_c^{|S|-2}( \mathcal{S}_{e(y,{\bf R}),{\bf R}}) \rightarrow H_c^{|S|-1}(W_{S-w,(\max_{\epsilon \in E_y}\{s^\epsilon_{\bf Q}\})})$$
  so that $\delta_y = \oplus _{e \in E_y} \delta_{y,e}$.

\begin{lemma}
Suppose that $H^{|S|-2}_c(W_{S-w,( r_{\bf Q})} )=0$ for any sufficiently large sequence $( r_{\bf Q})_{{\bf Q}\in \mathcal{P}}$.
If $\sum  _{{\bf R}\in {K_y}} v_{\bf R} \in \oplus _{{\bf R} \in K_{y}} H_c^{|S|-2}( \mathcal{S}_{e(y,{\bf R}),{\bf R}}) $ and $\delta _y (\sum v_{\bf R})=0$, then $\delta _y ( v_{\bf R})=0$ for all ${{\bf R}\in {K_y}}$.
\end{lemma}

\begin{proof} Since $H^{|S|-2}_c(W_{S-w,(\max_{\epsilon \in E_y} s^\epsilon_{\bf Q})})$  and $H^{|S|-2}_c(W_{S-w,(s^{e(y,{\bf Q})}_{\bf Q})})$ are trivial by assumption,
we have the following portion of the Mayer-Vietoris sequence for the pair from Lemma~\ref{l:mvs2}.
$$0 \rightarrow   \oplus  H_c^{|S|-2}(\mathcal{J}_{y,e(y,{\bf R}),{\bf R}})    \rightarrow    \oplus H_c^{|S|-2}( \mathcal{S}_{e(y,{\bf R}),{\bf R}}) \rightarrow H_c^{|S|-1}(W_{S-w,(\max_{\epsilon \in E_y}\{s^\epsilon_{\bf Q}\})})  $$
where $\delta _y$ is the rightmost map on the line above.
Thus, if $\delta _y (\sum v_{\bf R})=0$, then $\sum v_{\bf R} \in  \oplus  H_c^{|S|-2}(\mathcal{J}_{y,e(y,{\bf R}),{\bf R}})$, and in particular, for each $\bf R$ we have $v_{\bf R} \in   H_c^{|S|-2}(\mathcal{J}_{y,e(y,{\bf R}),{\bf R}})$ so that $\delta_y(v_{\bf R})=0$ for each $\bf R$.
\end{proof}

\begin{lemma}\label{l:nUFLQHO23}
Suppose that $H^{|S|-2}_c(W_{S-w,( r_{\bf Q})} )=0$ for any sufficiently large sequence $( r_{\bf Q})_{{\bf Q}\in \mathcal{P}}$. Let $y \in X_w$ be a vertex, and suppose $x \in H_c^{|S|-1}(F_y)$ is nonzero. Then there is at most one $e \in E_y$ such that $\rho_{y,e}(x)=0$.
\end{lemma}

\begin{proof}
Suppose that $\rho_{y,e}(x)=0$, and let $\epsilon \in E_y$. We will show that $\rho_{y,\epsilon}(x)= 0$ implies $e=\epsilon$, thus proving the lemma.
Applying the Mayer-Vietoris sequence to the sets from  Lemma~\ref{l:mvs}, and using Lemma~\ref{l:Jtriv},
we have 
\begin{align*}
\oplus _{{\bf R}\in K_{y,e}} H_c^{|S|-2}( \mathcal{S}_{e,{\bf R}}) \rightarrow H_c^{|S|-1}(W_{S-w,(\max_{\epsilon \in E_y}\{s^\epsilon_{\bf Q}\})})
\rightarrow H^{|S|-1}_c( W_{S-w,(s^{e}_{\bf Q})} )
\end{align*}
where the map on the left is $\delta _{y,e}$ and the map on the right is $\rho_{y,e}$. Therefore, $\rho_{y,e}(x)=0$ implies $x=\delta_{y,e}(\sum v_{\bf R})=\delta_{y}(\sum v_{\bf R})$ for some $\sum v_{\bf R}$.

Now if $\rho_{y,\epsilon}(x)=0$, then similarly, $x=\delta_{y}(\sum w_{\bf R})$ for some $\sum w_{\bf R}$. Therefore, $\delta _y (\sum(v_{\bf R}-w_{\bf R}))=x-x=0$ which implies that $\delta_y(v_{\bf R}-w_{\bf R})=0$ for each $\bf R$ by the previous lemma. 

Now fix some $\bf R$ with $\delta _{y,e} (v_{\bf R})\neq 0$. Then $$\delta_{y,\epsilon}(w_{\bf R})=\delta_{y}(w_{\bf R})=\delta_{y}(v_{\bf R})=\delta_{y,e}(v_{\bf R})\neq 0$$ from which we deduce that ${\bf R}$ is contained in $K_{y,e}$ and $K_{y,\epsilon}$. Thus,  $e=\epsilon$.
 \end{proof}

We are now ready to prove the main result of this section.

\begin{proposition}\label{p:397}
If $( r_{\bf Q})_{{\bf Q}\in \mathcal{P}}$ is sufficiently large, then $H_c^{k}(W_{S,(r_{\bf Q})})=0$ if $k \neq |S|$ and thus $H_c^{k}(X_{S,(r_{\bf Q})})=0$ if $k \neq |S|$.
\end{proposition}

\begin{proof}
If $|S|=1$, then we have proved this lemma in Lemma~\ref{l:947hdbk3129p}, so we assume the lemma is true for $S-w$ and prove it is true for $S$.

By Lemma~\ref{l:398268}, and since for any vertex $y\in X_w$ we have $F_y=\cup_{e \in E_y}F_e$, we see that  the collection $\{F_y\}$ taken over all vertices $y \in X_w$ is a cover for $W_{S,(r_{\bf Q})}$.  Note also that if $y$ and $z$ are the endpoints of an edge $e \se X_w$, then $F_e=F_y \cap F_z$. Thus, the nerve of $\{F_y\}$ can be identified with $X_w$, and there's an associated spectral with $E_2^{pq}=H^p_c(X_w,\{H^q_c(F_*)\})$ that converges to $H^{p+q}_c(W_{S,(r_{\bf Q})})$. (See, e.g. \cite{Brown} VII.4 for the analogous homology sequence. The derivation of the sequence we use here is a straightforward adaptation of that one.)

Since $F_e=e\times W_{S-w,(s^e_{\bf Q})}$, our induction hypothesis implies that   $H_c^{q}(F_e)=0 $ for $q \neq |S|-1$ and, together with $X_w$ being 1-dimensional, that implies $H^{p}_c(X_w,\{H^{q}_c(F_*)\})=0$ if $q \neq |S|-1$ or if $p\geq 2$. Thus, we will have $H_c^{k}(W_{S,(r_{\bf Q})})=0$ for $k\neq |S|$ if  $H^0_c(X_w,\{H^{|S|-1}_c(F_*)\})=0$, so we will verify that the kernel of the map 
$$d : \oplus_{y\in X_w^{(0)}}  H_c^{|S|-1}(F_y) \rightarrow \oplus_{e\in X_w^{(1)}}  H_c^{|S|-1}(F_e)$$
is trivial.

To do this, suppose $\sum g_y \in \oplus  H_c^{|S|-1}(F_y)$ is nonzero.
Choose some vertex $y \in X_w$ with $g_y \neq 0$ and such that $y$ is contained in an edge $e$ and in the component of  $X_w - e^\circ$ all of whose vertices $y'\neq y$ have $g_{y'}=0$.

By Lemma~\ref{l:nUFLQHO23}, there is an edge $\epsilon \in E_y - e$ such that  $\rho_{y,\epsilon}(g_y)\neq 0$, Therefore, the $H_c^{|S|-1}(F_\epsilon)$ component of $d(\sum g_y)$ is nonzero, and thus $d(\sum g_y)\neq 0$.

We have seen that $H^{p}_c(X_w,\{H^{q}_c(F_*)\})=0$ if $(p,q)\neq (1,|S|-1)$.
\end{proof}

If the sequence $( r_{\bf Q})_{{\bf Q}\in \mathcal{P}}$ from Proposition~\ref{p:397} is constant, then we have the immediate

\begin{corollary}
If $r \gg 0$ and $k \neq |S|$, then $H_c^{k}(X_{S,r})=0$.
\end{corollary}

The proof of Proposition~\ref{p:397} given above applies to horosphere complements in products of trees that are more general than those arising from arithmetic groups. In particular,
suppose $d \in \mathbb{N}$ and that $T_i$, for $1 \leq i \leq d$, is a locally finite tree with no vertices of valence 1. Choose geodesics $\Sigma_i : \mathbb{R} \rightarrow T_i$ pararmeterized such that the integer values of $\Sigma_i$ are exactly the vertices in $T_i$ in the image of $\Sigma_i$. For any collection of positive numbers $\lambda _i $, let $\beta : \prod _{i=1}^d T_i \rightarrow \mathbb{R}$ be the Busemann function for $\rho :\mathbb{R} \rightarrow \prod _{i=1}^d T_i$ given by $\rho(t)=(\Sigma_i(\lambda_i t))_{i=1}^n$. Let $Z=\beta^{-1}((-\infty,r])$ for any given $r \in \mathbb{R}$.

\begin{corollary} \label{cor:hc_single_horoball_complement}
If $k \neq d$, then $H_c^{k}(Z)=0$.
\end{corollary}

\begin{proof}
Apply Proposition~\ref{p:397} to a sequence $( r_{\bf Q})_{{\bf Q}\in \mathcal{P}}$ that has exactly one finite value, and the result is the statement of this corollary. The only exception is that Proposition~\ref{p:397} applies to trees whose valences are dictated by an arithmetic group, and it applies to Busemann functions for rays whose slopes (the $\lambda _i$) are determined by an arithmetic group. But neither of these explicit data are used in the proof of  Proposition~\ref{p:397}.
\end{proof}

\section{Topology of horospheres} \label{sec:horosphere}

Let $X=\prod_{i=1}^d T_i$ where
each $T_i$ is a locally finite tree with no vertices of valence $1$. Suppose each edge length in $T_i$ equals 1. For each tree $T_i$,
choose a geodesic $\Sigma_i \se T_i$ and label its vertices $x_{i,n}$ for
$n\in \Z$. This induces a height function $h_i$ on the vertices of
$T_i$ where $h_i(x_{i,0}) = 0$ and $h_i(v) = n-d(v,x_{i,n})$ if the closest vertex
of $\Sigma_i$ is $x_{i,n}$. Extend each $h_i$ linearly over edges to produce a height
function $h_i$ defined on all of $T_i$. For $1\leq i \leq d$, we choose $\lambda _i>0$ and we define a Busemann function $\beta: X\to \R$ by $\beta(x_1\dotsc, x_d) = \sum_i \lambda _i h_i(x_i)$.

Say that a vertex $v\in T_i$ is {\em below} a vertex $w\in T_i$ if
there is a path $\gamma$ from $v$ to $w$ such that $h_i\circ \gamma$
is strictly increasing. In this case we say $w$ is {\em above}
$v$. Note that for any $x_i\in T_i$ and $t > 0$ there is a unique
point $y_i\in T_i$ above $x_i$ such that $h_i(y_i) = h_i(x_i) + t$.
Using this notation, the assignment $x_i \mapsto y_i$ defines a flow
$\phi_{i,t} : T_i\to T_i$. These then define a
flow on $X$ by 
\[
 \phi_t(x_1,\dotsc, x_d) = \left( \phi_{1,t/(\lambda_1\sqrt{d})}(x_1), \dotsc,
 \phi_{d,t/(\lambda_d\sqrt{d})}(x_d) \right)
\]
Note that $\beta( \phi_t(x) ) = \beta(x) + t$.

For $r \in \mathbb{R}$, we define
\begin{align*}
Y_r &= \beta^{-1}(r)\\
X_r &= \beta^{-1}(-\infty, r],\text{ and}\\
B_r &= \beta^{-1}[r, \infty)
\end{align*}

The space $X$ naturally has the structure of a cube complex. Subdivide
this structure to give $X$ the structure of a cell complex such that
$Y_r$ and $X_r$ are subcomplexes. In particular, for each $(d-1)$-cell
$e$ of $Y_r$ there is a unique $d$-cell $\hat e$ of $X$ lying above
$e$ such that $\overset{\circ}{e} \cap \hat e \neq \emptyset$.

In this section, all cohomology groups will be understood to have
coefficients in some ring $R$.

\subsection{Horoball cohomology}

\begin{lemma} \label{l:horoball_cohom}
 For all $r$ and $k$, $H_c^k(B_r) = 0$.
\end{lemma}

\begin{proof}
  For any number $m > n$ let
 \[
 C(m)=\{\, x = (x_1,\dotsc, x_d) \in X \mid \beta(x)\geq r \text{ and $x_i$ lies below
   $x_{i,m}$} \,\}
 \]
 and let
 \[
 \partial ^\uparrow C (m)=\{\, x \in C(m) \mid h_i(x) = m \text{ for some } i\,\}
 \]
 Note that $C(m)$ deformation retracts onto  $\partial ^\uparrow C(m)$. Thus we have that $H^k( C(m), \partial^\uparrow C(m) )=0$.

  Note that the sets $C(m)$ form an exhaustion of $B_r$ by compact
  sets. Let $c(m)$ be the closed subset of $C(m)$ consisting of points whose distance from $\partial ^\uparrow C(m)$ is at least $\varepsilon$, for some small $\varepsilon>0$.
   The compact sets $c(m)$ also
  form an exhaustion of $B_r$, so it suffices to show
  $H^k( B_r, B_r - c(m)) = 0$ for all $m$. By excision we have 
  \[
  H^k(B_r, B_r-c(m) ) \cong H^k( C(m), C(m)- c(m))
  \]
  Because $C(m) - c(m)$ deformation retracts onto $\partial^\uparrow C(m)$, we have
  \[
  H^k( C(m), C(m)-c(m)) \cong H^k( C(m), \partial^\uparrow C(m) )=0
  \]
\end{proof}

\subsection{Horosphere cohomology}

For $n \in \mathbb{N}$, the collection
$\{ H_c^{d-1}(Y_n) \}_{n\in \N}$ forms a directed system under the
maps $(\phi_1)^*: H_c^{d-1}(Y_{n+1}) \to H_c^{d-1}(Y_n)$, where
$\phi_1$ is the time 1 flow on $X$. The goal of this section is to
show $\varprojlim_n H_c^{d-1}(Y_n)$ is trivial and $\varprojlim^1_n
H_c^{d-1}(Y_n)$ is torsionfree.

There is a Mayer-Vietoris exact sequence
\[
\dotsb \to H_c^{d-1}(X_n) \oplus H_c^{d-1}(B_n) \to H_c^{d-1}(Y_n) \to
H_c^{d}(X) \to \dotsb
\]

We know $H_c^{d-1}(X_n) = 0$ by Corollary \ref{cor:hc_single_horoball_complement} and
$H_n^{d-1}(B_n) = 0$ by Lemma \ref{l:horoball_cohom}. Therefore the
connecting map $H_c^{d-1}(Y_n) \to H_c^d(X)$ is injective. In this way
we consider each module $H_c^{d-1}(Y_n)$ as a submodule of
$H_c^d(X)$. Note that $(\phi_1)^* : H_c^d(X) \to H_c^d(X)$ is the
identity map. It therefore follows from naturality of the
Mayer--Vietoris sequence that
$(\phi_1)^* : H_c^{d-1}(Y_{n+1}) \to H_c^{d-1}(Y_n)$ is the inclusion
map of subgroups of $H_c^d(X)$.

We will prove $\varprojlim_n H_c^{d-1}(Y_n) = 0$, for which we
set up notation. For any vertex $v\in T_i$ let $g_i(v)$ be the unique
vertex above $v$ such that $h_i(g_i(v)) = h_i(v)+1$. Under the
identification of $X^{(0)}$ with $\prod_i T_i^{(0)}$, let
$g: X^{(0)} \to X^{(0)}$ be the function defined to be $g_i$ in the
$i^{th}$ coordinate, so that $g(w)(i) = g_i(w(i))$ for any
$w\in \prod_i T_i^{(0)}$.

Given $i$ and $n$, let $C_{i,n} \subseteq T_i$ be the subtree of $T_i$
spanned by the set of vertices
\[ 
\left \{ v\in T_i \suchthat \text{$v$ is below $x_{i,n}$ and $h_i(v)\geq
  -n$} \right \}
\]
Let
$K_n = \prod_{i=1}^d C_{i,n}$. The collection
$\{ K_n \}_{n=0}^{\infty}$ forms an exhaustion of $X$ by compact sets,
and so $H_c^k(X) = \varinjlim_n H^k( X, X\setminus K_n).$

We compute each relative cohomology group $H^k(X, X\setminus K_n)$ as
the cohomology of the quotient space $X/(X\setminus K_n)$. Let
$EC_{i,n}$ be the set of vertices $v\in C_{i,n}$ such that
$h_i(v) = -n$. Each set $C_{i,n}$ is homotopy equivalent relative $EC_{i,n} \cup \{x_{i,n}\}$ to a star with
$\#(EC_{i,n})$ leaves, and so if $\partial K_n$ is comprised of points in $K_n$ whose $i^{th}$ coordinate for some $i$ is contained in $EC_{i,n} \cup \{x_{i,n}\}$, then $K_n$ is homotopy equivalent relative $\partial K_n$ to a cube
complex with a top dimensional cube for each vertex in
$\prod_{i=1}^d EC_{i,n}$. It follows that there is a homotopy
equivalence
\[
X/ (X\setminus K_n) \simeq  \bigvee_{v\in \prod_{i=1}^d EC_{i,n}} S^d
\]
To simplify notation, let $\Lambda_n$ be the set of vertices in
$\prod_{i=1}^d EC_{i,n}$. With this notation, it follows that there is
an isomorphism 
\[
H^k(X, X\setminus K_n) \cong
\begin{cases} 
R^{\Lambda_n} &\text{ if }k=d\\
0 &\text{ else.}
\end{cases}
\]
Under this identification, the map
$f_n : R^{\Lambda_n} \to R^{\Lambda_{n+1}}$ induced by the map of
pairs $(X, X\setminus K_{n+1}) \to (X, X\setminus K_n)$ is described
as follows: Given a function $\alpha : \Lambda_n \to R$,
define $f_n(\alpha)(w) = \alpha(g(w))$ if
$g(w) \in \Lambda_n$ and $f_n(\alpha)(w) = 0$
otherwise.

Given a vertex $v\in \Lambda_n$, let
$\overline{x_{i,n}, v(i)}$ denote the geodesic segment between
$x_{i,n}$ and $v(i)$, and let $F_v$ be the cube
$\prod_{i=1}^d \overline{x_{i,n}, v(i)}$. Note there is an equality of
spaces $K_n = \bigcup_{v\in \Lambda_n} F_v$. Given a compactly
supported cellular cochain $\phi\in Z_c^d(X)$, the assignment
$[\phi] \mapsto \left( v\mapsto \phi(F_v) \right)_n$ gives the
isomorphism $H_c^d(X) \cong \varinjlim_n R^{\Lambda_n}$.

Note that the above is a proof of the well-known

\begin{proposition} \label{prop:hctreeproduct}
 $H_c^k(X)=0$ if $k \neq d$ and $H_c^d(X)$ is a free $R$-module.
\end{proposition}

See Borel-Serre \cite{B-S2} for a more general theorem about the compactly supported cohomology of Euclidean buildings.

We now observe that $H_c^*(Y_r)$ is concentrated in dimension $d-1$.

\begin{proposition} \label{prop:horosphere-cohom}
 $H_c^k(Y_r)=0$ if $k \neq d-1$ and $H_c^{d-1}(Y_r)$ is a free $R$-module.

\end{proposition}
\begin{proof}
  There is a Mayer-Vietoris exact sequence
  \[
    \dotsb \to H_c^{k}(X_r) \oplus H_c^{k}(B_r) \to H_c^{k}(Y_r) \to
    H_c^{k+1}(X) \to \dotsb
  \]
  By Corollary \ref{cor:hc_single_horoball_complement} we know
  $H_c^k(X_r) = 0$ for $k\leq d-1$. By Lemma \ref{l:horoball_cohom} we
  know $H_c^k(B_r) = 0$ for all $k$. And by Proposition
  \ref{prop:hctreeproduct} we know $H_c^{k+1}(X) = 0$ for $k\leq
  d-2$. The result follows.
\end{proof}

Using the notation that we established prior to Proposition~\ref{prop:hctreeproduct}, we will prove 

\begin{lemma} \label{horosphere_lim0}
  $\varprojlim_n H_c^{d-1}(Y_n) = 0$.
\end{lemma}
\begin{proof}
  Take any cohomology class $[\phi] \in H_c^d(X)$. Choose $n\in \N$ so
  that the support of $\phi$ is contained in $K_n$. 

  Consider any vertex $v\in \Lambda_n$, and choose $m$ such that $m >\beta(K_{n+1})$.
  Suppose
  $[\phi] = [\delta \psi]$ for some $\psi\in H_c^{d-1}(Y_m)$, where
  $\delta$ is the chain map inducing the connecting homomorphism in
  the Mayer-Vietoris sequence. Choose $N>n+1$ such that the support of
  $\psi$, and hence also the support of $\delta \psi$, is contained in
  $K_N$. Because $[\phi]$ and $[\delta \psi]$ are equal in
  $H_c^d(X)$ their images in $R^{\Lambda_N}$ are equal.

  Choose any $w\in \Lambda_N$ so that $g^{N-n}(w) = v$. Then $F_v
  \subseteq F_w$. Since the suport of $\phi$ is contained in $K_n$ and
  $F_v = F_w\cap K_n$ we have $\phi(F_v) = \phi(F_w)$. 

  For each $1\leq i \leq d$ choose a vertex $e_i \in EC_{i, N}$ such
  that $g_i^{N-n-1}(e_i) \in EC_{i, n+1}$ but $g_i^{N-n} \notin EC_{i,
    n}$. Let $P(d)$ be the set of subsets of $\{1,\dotsc, d\}$. For
  each $\sigma \in P(d)$ define $w_\sigma \in \Lambda_n$ by
  \[
  w_\sigma(i) = \begin{cases} w(i) & \text{if }i\notin \sigma \\
    e_i& \text{if } i\in \sigma.\end{cases}
  \]
  Note that $w_\emptyset = w$. If $\sigma \neq \emptyset$ then
  $F_{w_\sigma} \cap K_n = \emptyset$ and so $\phi(F_{w_\sigma}) = 0$.
  Because $[\phi]$ and $[\delta\psi]$ have the same image in
  $R^{\Lambda_N}$ we see
  $\phi(F_{w_\sigma}) = \delta\psi( F_{w_\sigma} )$ for all
  $\sigma \in P(d)$.

  We claim the $(d-1)$-chain
  $\sum_{\sigma \in P(d)} (-1)^{|\sigma|} (F_{w_\sigma} \cap Y_m )$ is
  the zero chain. Indeed, for $1\leq i \leq d$, let $u_i$ be an order 2 isometry of $T_i$ with $u_i(w(i))=e_i$ and $u_i(x_i)=x_i$ if $h_i(x_i)\geq n+1$. For $\sigma \in P(d)$ we let $u_\sigma$ be the product of $u_i$ with $i \in \sigma$. In particular, $u_\emptyset =1$. Notice that $u_\sigma F_w = F_{w_\sigma}$, and that $u_\sigma Y_m =Y_m$.

  We let $P(d)^*=P(d)-\{1,\cdots,d\}$, and for each $\tau \in P(d)^*$, we define $$R_\tau =\{(x_i) \in F_w \cap Y_m \mid h_i(x_i)\leq n+1 \text{ precisely when } i\in \tau \}$$
 Thus, $\cup _{\tau \in P(d)^*} R_\tau =  F_w \cap Y_m$, if $\tau_1 \neq \tau_2$ then $R_{\tau_1}$ and $R_{\tau_2}$ do not contain a common $(d-1)$-cell, and if $\sigma \in P(d)$ then $u_{\sigma} R_\tau =u_{\sigma \cap \tau}R_\tau$.
 
 Recall that by the binomial theorem, if $k \in \mathbb{N}$, then
 $\sum_{\mu \in P(k)} (-1)^{|\mu|} = 0$. Thus we have 
 \begin{align*}
 \sum_{\sigma \in P(d)} (-1)^{|\sigma|} (F_{w_\sigma} \cap Y_m )
 & =  \sum_{\sigma \in P(d)} (-1)^{|\sigma|} (u_\sigma F_{w} \cap Y_m ) \\
& =  \sum_{\sigma \in P(d)} (-1)^{|\sigma|} u_\sigma( F_{w} \cap Y_m ) \\
 &=\sum_{\sigma \in P(d)} (-1)^{|\sigma|} u_\sigma\Big(\sum_{\tau \in P(d)^*} R_\tau \Big) \\
  &=\sum_{\tau \in P(d)^*} \sum_{\sigma \in P(d)} (-1)^{|\sigma|}  u_\sigma R_\tau  \\
 &=\sum_{\tau \in P(d)^*} \sum_{\sigma \in P(d)} (-1)^{|\sigma|}  u_{\sigma \cap \tau} R_\tau  \\
  &=\sum_{\tau \in P(d)^*} \sum_{\rho \in P(|\tau|)} \sum_{\mu \in P(d-|\tau|)} (-1)^{|\rho|+|\mu|}  u_{\rho} R_\tau  \\
   &=\sum_{\tau \in P(d)^*} \sum_{\rho \in P(|\tau|)} (-1)^{|\rho|} \sum_{\mu \in P(d-|\tau|)} (-1)^{|\mu|}  u_{\rho} R_\tau  \\
    &=\sum_{\tau \in P(d)^*} \sum_{\rho \in P(|\tau|)} (-1)^{|\rho|}\; 0 \;  u_{\rho} R_\tau  \\
  &  =0
 \end{align*}
  
  This establishes our claim that
  $\sum_{\sigma \in P(d)} (-1)^{|\sigma|} (F_{w_\sigma} \cap Y_m )$ is
  the zero chain.
  
  Using the definition of the connecting map $\delta$
  we therefore have
  \begin{align*}
    0 &= \psi \left( \sum_{\sigma \in P(d)} (-1)^{|\sigma|} F_{w_\sigma}
      \cap Y_\sigma \right) \\ 
      &= \sum_{\sigma \in P(d)} (-1)^{|\sigma|} \delta \psi( F_{w_\sigma}
        ) \\
      &= \sum_{\sigma \in P(d)} (-1)^{|\sigma|} \phi( F_{w_\sigma}) \\
      &= \phi( F_w)\\
      &= \phi(F_v)
  \end{align*}
  This shows that if $[\phi]\in H_c^d(X)$,  if  $[\phi] $ is contained  in $ H_c^{d-1}(Y_m) \leq H_c^d(X)$ for some sufficiently large value of $m$, then
  $[\phi] = 0$. That is, $\varprojlim_n H_c^{d-1}(Y_n)=\cap_m H_c^{d-1}(Y_m) =0$.
\end{proof}

Our next goal is to prove that $\varprojlim^1 H_c^{d-1}(Y_m)$ is
torsion-free. 
\begin{lemma} \label{cohomdivision} Suppose
  $[\psi] \in H_c^{d-1}(Y_m)$ and suppose there are
  $[\phi]\in H_c^d(X)$ and $r\in R$ such that
  $r[\phi] = [\delta \psi]$. Then there is some
  $[\tilde \phi] \in H_c^{d-1}(Y_m)$ such that
  $r[\tilde \phi] = [\psi]$.
\end{lemma}
\begin{proof}
  Let $\hat Y_m$ be the subcomplex of $X$ containing only
  the (subdivided) $d$-cells $e \se X$ such that $\beta(e)\geq m$ and $e\cap Y_m \neq \emptyset$.
  Let $i: \hat Y_m\to X$ be the inclusion map. The Mayer-Vietoris
  connecting map induces a homomorphism
  $\hat \delta : H_c^{d-1}(Y_m) \to H_c^d(\hat Y_m)$ such that the
  following diagram commutes:
  \[ 
  \xymatrix{ & H_c^{d-1}(Y_m) \ar[dl]_\delta \ar[dr]^{\hat\delta} & \\
    H_c^d(X) \ar[rr]^{i^*}& & H_c^d(\hat Y_m) } 
  \]
  Given any $(d-1)$-cell $c\subseteq Y_m$, let $\hat c \subseteq \hat Y_m$ be the unique
  $d$-cell in $\hat Y_m$ with $c \subseteq {\hat c}$.
  
  If 
  $[\phi]\in H_c^d(\hat Y_m)$, then define $[\epsilon \phi]\in H_c^{d-1}(Y_m)$ by $\epsilon\phi(c) = \phi(\hat c)$.
  Then 
  $\epsilon : H_c^d(\hat Y_m)
  \to H_c^{d-1}(Y_m)$ is the inverse of $\hat \delta$, so there
  is an isomorphism $H_c^{d-1}(Y_m) \cong H_c^d(\hat Y_m)$. The lemma
  follows by setting $[\tilde \phi] = \epsilon i^* ( [\phi])$.
\end{proof}

\begin{lemma} \label{horosphere_lim1tf}
  $\varprojlim^1 H_c^{d-1}(Y_m)$ is torsionfree as an $R$-module.
\end{lemma}
\begin{proof}
  Recall that $\varprojlim^1 H_c^{d-1}(Y_m)$ is the cokernel of the map 
  $$\Delta : \prod _m H_c^{d-1}(Y_m) \rightarrow \prod _m H_c^{d-1}(Y_m)  $$
  where $\Delta\big(([\psi_m])_m\big)=([\psi_m]-[\psi_{m+1}])_m$.
  
  Suppose $( [\xi_m])_m \in \prod_m H_c^{d-1}(Y_m)$ and that there is some
  regular element $r\in R$ such that $( r [\xi_m] )_m$ is in the
  image of $\Delta$. Let $([\psi_m])_m \in \prod_m H_c^{d-1}(Y_m)$ be
  a sequence such that $r [\xi_m] = [\psi_m] - [\psi_{m+1}]$ for all
  $m$. Note that this implies  that for any $M > m$ there is some
  $\zeta_{m,M} \in H_c^{d-1}(Y_m)$ so that
  $r [\zeta_{m,M}] = [\psi_m] - [\psi_M]$.

  Fix any $m$ and choose $n\in \N$ such that $\psi_m$ is supported
  on $K_{n}$. Take any $v\in \Lambda_{n}$. Choose any $M>m$ such
  that $M>\beta(K_{n})$. Choose $N\in \N$ such that $\psi_M$
  is supported on $K_N$.

  Choose any $w\in \Lambda_N$ such that $g^{N-n}(w) = v$. Construct
  vertices $w_\sigma \in \Lambda_N$ for each $\sigma\in P(d)$ as in the
  proof of Lemma \ref{horosphere_lim0}, so that $F_{w_\emptyset} = F_v$
  and $\psi_m( F_{w_\sigma} ) = 0$ if $\sigma \neq \emptyset$ and
  $\sum_{\sigma \in P(d)} (-1)^{|\sigma|} F_{w_\sigma} \cap Y_M$ is
  the zero chain. Then
  \[ \psi_M \left( \sum_{\sigma \in P(d)} (-1)^{|\sigma|}
    F_{w_\sigma} \cap Y_M \right) = 0 \]
so
 \[ \delta\psi_M \left( \sum_{\sigma \in P(d)} (-1)^{|\sigma|}
    F_{w_\sigma} \right) = 0 \]

  It follows that $r$ divides the quantity
  \begin{align*}
  (\delta \psi_m - \delta \psi_M)&\left( \sum_{\sigma \in P(d)}
    (-1)^{|\sigma|} F_{w_\sigma} \right)\\
                                 &= \delta\psi_m \left( \sum_{\sigma \in P(d)}
    (-1)^{|\sigma|} F_{w_\sigma} \right) - \delta\psi_M \left( \sum_{\sigma \in P(d)}
    (-1)^{|\sigma|} F_{w_\sigma} \right)\\
  &= \delta\psi_m(F_{w_\emptyset})\\
  &= \psi_m(F_v  \cap Y_M)
  \end{align*}

  This holds for any vertex $v\in \Lambda_n$, so $r$ divides the image
  of $[\psi_m]$ in $R^{\Lambda_{n}}$. Therefore $r$ divides the image of
  $[\psi_m]$ in $H_c^d(X)$. By Lemma \ref{cohomdivision}, for each $m$
  there is some $[\phi_m] \in H_c^{d-1}(Y_m)$ such that $r [\phi_m] =
  [\psi_m]$. It follows that $([\xi_m])_m = \Delta( [\phi_m] )$, which
  completes the proof.
\end{proof}

\section{Examples of semiduality groups}
\label{sec:examples}

Below we provide examples of semiduality groups. In order to verify
the condition on the cohomological dimension, we recall the following
standard result.

\begin{lemma} \label{lem:cdlemma} Suppose $\Lambda$ is a group acting
  on an acyclic cell complex $X$ with finite cell stabilizers. Suppose
  $R$ is a commutative ring such that $\abs{\Lambda_\sigma}$ is
  invertible in $R$ for any cell stabilizer $\Lambda_\sigma$. Then
  $\cd_{R}\Lambda \leq \dim(X)$.
\end{lemma}
\begin{proof}
  Suppose $M$ is an $R\Lambda$-module. For each $j$ let $\Sigma_j$ be a
  set of representatives of $\Lambda$-orbits of $j$-cells of $X$. There
  is a spectral sequence (compare to the homology version appearing in 
  \cite[VII.7.7, p173]{Brown})
  \[
    E_1^{jq} = \prod_{\sigma\in \Sigma_j} H^q(\Lambda_\sigma; M_\sigma) \implies
    H^{j+q}(\Lambda; M)
  \]
  For any $q>0$ the module $H^q(\Lambda_\sigma, M_\sigma)$ is annihilated by
  $\abs{\Lambda_\sigma}$. But $M_\sigma$ is an $R$-module and
  $\abs{\Lambda_\sigma}$ is invertible in $R$ for any cell $\sigma$,
  so the groups $H^q(\Lambda_\sigma; M_\sigma)$ are trivial if $q>0$. By
  definition, the groups $E_1^{jq}$ are trivial for $j > \dim (X)$. It
  follows that $\cd_{R}\Lambda \leq \dim(X)$.
\end{proof}

\subsection{Rank 1 arithmetic groups}

In this section we prove Theorem \ref{t:mt}. To that end, suppose
$\mathcal{O}_S$ is the ring of $S$-integers in a global function field
$K$ of characteristic $p$. Suppose $\bf G$ is a noncommutative,
absolutely almost simple algebraic $K$-group and $\Gamma$ is a
finite-index subgroup of $\bfg(\mco_S)$ such that any torsion element
of $\Gamma$ is a $p$-element. 

\begin{proposition} 
  $\Gamma$ is a $\Z[1/p]$-semiduality group of dimension $k(\bfg, S)$.
\end{proposition}
\begin{proof}
  $\Gamma$ acts on the product of trees $X_S$ with finite $p$-group
  stabilizers. It follows from Lemma \ref{lem:cdlemma} that
  $\cd_{\Z[1/p]}\Gamma \leq \dim(X_S) = k({\bf G}, S)$.  It is known
  that $\Gamma$ is type $FP_{k({\bf G}, S)-1}$ over $\Z[1/p]$ (see
  \cite{St}, \cite{B-W2}, \cite{B-K-W}).  It remains to show that
  $H^*(\Gamma, \Z[1/p]\Gamma)$ is concentrated in dimension
  $k({\bf G}, S)$, where it is flat as a $\Z[1/p]$-module.

  For sufficiently large $n\in \N$, $\Gamma$ acts properly cocompactly
  on each $X_{S,n}$ by Lemma \ref{l:prop-cocpt}. As $n$ tends to
  infinity, the spaces $X_{S,n}$ exhaust $X_S$. By Proposition
  \ref{prop:ourcohomology} there is an isomorphism
  $\hcu{k}(X_S) \cong H^k(\Gamma, \Z[1/p]\Gamma)$, where here and for
  the rest of the proof we take cohomology of spaces with $\Z[1/p]$
  coefficients.  We proved in Proposition \ref{p:397} that
  $H_c^*(X_{S,n})$ is concentrated in dimension $d$. It follows from
  Proposition \ref{prop:hcu_ex} that $\hcu{*}(X_S)$ is concentrated in
  dimension $d$, where $\hcu{d}(X_S) = \varprojlim_n H_c^d( X_{S,n})$.

  It remains only to show that $\varprojlim_n H_c^d( X_{S,n})$ is
  $\Z[1/p]$-torsion-free, since $\Z[1/p]$ is a principal ideal domain.
  The closure of the complement of $X_{S,n}$ is a
  disjoint union of horoballs
  $\bigcup_{{\bf Q}\in \mcp} B_{{\bf Q}, S, n}$. Up to
  proper homotopy, each set $B_{{\bf Q}, S, n}$ is a
  set of the form $B_n$ as defined in \S\ref{sec:horosphere}, and
  $X_{S,n} \cap B_n = Y_n$, where $Y_n = Y_{{\bf Q}, S, n}$.  There is a Mayer-Vietoris exact sequence
  \begin{align*}
    H_c^{d-1}(X_{S,n}) \oplus &H_c^{d-1}\big( \bigcup_{\mcp} B_n \big)
                            \to H_c^{d-1} \big( \bigcup_{\mcp} Y_n
                            \big) \to \\ 
                          &H_c^{d}(X_S) \to H_c^d(X_{S,n}) \oplus
                            H_c^d\big( \bigcup_\mcp B_n \big) \to
                            H_c^d\big(\bigcup_\mcp Y_n\big)
  \end{align*}
  Because the unions are disjoint, for each $k$ there are isomorphisms
  \[
    H_c^k\big( \bigcup_\mcp B_n \big) \cong \bigoplus_\mcp H_c^k(B_n)
    \quad \text{ and } \quad 
    H_c^k\big( \bigcup_\mcp Y_n \big) \cong \bigoplus_\mcp H_c^k(Y_n)
  \]
  We know $H_c^k(B_n) = 0$ for all $k$ by Lemma
  \ref{l:horoball_cohom}. We also know $H_c^{d-1}(X_{S,n}) = 0$ by
  Proposition \ref{p:397}. Clearly $H_c^d(Y_n) = 0$ since $Y_n$ is
  $(d-1)$-dimensional. Therefore we have a short exact sequence
  \[
    0 \to \bigoplus_\mcp H_c^{d-1}(Y_n) \to H_c^d(X_S) \to H_c^d(X_{S,n}) \to 0
  \]
  These maps are compatible with the maps induced by inclusion
  $i_n : X_{S,n}\to X_{S, n+1}$, the time 1 flow $\phi_1 : Y_n \to Y_{n+1}$,
  and the identity map $X_S\to X_S$. The above short exact sequence
  therefore gives rise to a short exact sequence of codirected systems
  of compactly support cohomology, from which there is an exact
  sequence
  \[
    0\to \varprojlim \bigoplus_\mcp H_c^{d-1} (Y_n ) \to
    H_c^d(X_s) \to \varprojlim H_c^d(X_{S,n}) \to \varprojlim{}^1
    \bigoplus_\mcp H_c^{d-1} ( Y_n ) \to 0
  \]
  The maps of the system $\{ \oplus H_c^{d-1}(Y_n) \}$ preserve the
  direct sum structure. We know $\varprojlim H_c^{d-1}(Y_n)$ is
  trivial by Proposition \ref{horosphere_lim0} and
  $\varprojlim^1 H_c^{d-1}( Y_n )$ is torsionfree by Proposition
  \ref{horosphere_lim1tf}. Since $H_c^d(X_S)$ is torsionfree by
  Proposition \ref{prop:hctreeproduct}, it follows that
  $\varprojlim H_c^d(X_{S,n})$ is torsionfree by the short exact sequence
  \begin{equation} \label{eq:finales}
    0 \to H_c^d(X_s) \to \varprojlim H_c^d(X_{S,n}) \to 
    \bigoplus_\mcp \varprojlim{}^1 H_c^{d-1} ( Y_n ) \to 0
  \end{equation}
\end{proof}

This proves Theorem \ref{t:mt}, as every module in sequence
\eqref{eq:finales} is a $\Z[1/p]\bfg(K)$-module by Lemma
\ref{lem:commaction}.

\subsection{Solvable groups}

In this section we prove groups of the form $\Btwo(\mco_S)$ are
semiduality groups. We then prove that some generalizations of certain
groups of this form are also semiduality groups, namely lamplighter
groups, Diestel-Leader groups, and countable direct sums of finite
groups. All are straightforward applications of the following lemma.

\begin{lemma} \label{l:horoball-semiduality} Let $X$ be a product of
  $d$ trees with Busemann function $\beta: X\to \R$ as described in
  \S\ref{sec:horosphere}. Suppose a group $\Lambda$ acts on $X$
  cellularly, with finite cell stabilizers, and cocompactly on subsets
  of the form $\beta^{-1}(I)$ for closed intervals $I$. Suppose $R$ is a
  principal ideal domain such that $\abs{\Lambda_\sigma}$ is
  invertible for every cell stabilizer $\Lambda_\sigma$. Then
  $\Lambda$ is an $R$-semiduality group of dimension $d$.
\end{lemma}
\begin{proof}
  Define 
  \begin{align*}
    Y_n &= \beta^{-1}(\{n\})\\
    X_n &= \beta^{-1}[0, n],\text{ and}\\
    B_n &= \beta^{-1}[n, \infty)
  \end{align*}

  The space $X$ is contractible, so by Lemma \ref{lem:cdlemma} we
  know $\cd_{R}\Lambda \leq \dim(X) = d$. Since $\Lambda$ acts
  cocompactly with finite stabilizers on a horosphere $Y_n$ and
  $\tilde H_k(Y_n) = 0$ for $k < n-1$ by \cite[3.1]{bux}, Brown's
  criterion implies that $\Lambda$ is type $FP_{d-1}$ (see for example
  \cite[1.1]{brownfiniteness}).

  The complexes $X_n$ form an exhaustion of $B_0$ by closed,
  $\Gamma$-invariant sets such that $\Gamma \backslash X_n$ is
  compact. Therefore by the results of \S\ref{sec:trans_to_topology}
  there is an isomorphism $H^*(\Lambda; R\Lambda) \cong
  \hcu{k}(B_0)$ and for each $k$ there is an
  exact sequence
  \begin{equation}
    0\to \varprojlim{}^1 H_c^{k-1}(X_n) \to \hcu{k}(B_0) \to \varprojlim
    H_c^{k}(X_n) \to 0
  \end{equation}
  Note that the flow $\phi_t$ provides a
  proper deformation retraction of $X_n$ to $Y_n$ so
  $H_c^*(X_n)\cong H_c^*(Y_n)$. We know $H_c^*(Y_n)$ is concentrated
  in dimension $d-1$ by Proposition \ref{prop:horosphere-cohom}. Now
  Lemma \ref{horosphere_lim0} says $\varprojlim H_c^{d-1}(Y_n) = 0$ so
  $\hcu{*}(B_0)$ is concentrated in dimension $d$. In that dimension
  there is an isomorphism
  $\hcu{d}(B_0) \cong \varprojlim{}^1 H_c^{d-1}(Y_n)$, which is
  torsionfree by Lemma \ref{horosphere_lim1tf} and hence flat as an
  $R$-module.
\end{proof}

Suppose $\mathcal{O}_S$ is the ring of $S$-integers in a global
function field $K$ of characteristic $p$. Let $\Btwo$ be the group of
upper triangular matrices of determinant 1.

\begin{theorem}
  Suppose $\Gamma$ is a finite index subgroup of $\Btwo(\mco_S)$ such
  that the order of every finite order element is a power of $p$. Then
  $\Gamma$ is a $\Z[1/p]$-semiduality group of dimension $\abs{S}$.
\end{theorem}
\begin{proof}
  In the notation of \S\ref{s:red}, we may choose
  $\mathbf{P} = \Btwo$. Then applying Lemma \ref{l:horoball-action}
  with $\gamma = 1$ and $f=1$, we see that $\mathbf{P}$ acts on the
  horoball $B_{\mathbf{P},S,n}$ for all sufficiently large $n\in
  \N$. In fact $\mathbf{P}(\mco_S)$ is the entire stabilizer of
  $B_{\mathbf{P},S,n}$ in $\SLtwo(\mco_S)$ since if
  $\gamma \in \SLtwo(\mco_S)$ and $\gamma B_n = B_n$ then
  $B_{\mathbf{P}, S, n} = B_{\gamma \mathbf{P} \gamma^{-1}, S, n}$,
  which by Proposition \ref{p:prune2}(ii) means
  $\mathbf{P} = \gamma \mathbf{P} \gamma^{-1}$ and so
  $\gamma \in \mathbf{P}$. It follows that the action of $\mathbf{P}$
  on $Y_{\mathbf{P},S,n}$ is proper and cocompact since the action of
  $\SLtwo(\mco_S)$ is proper and cocompact on $X_S$. In particular,
  cell stabilizers are finite.

  $\Gamma$ acts on the product of trees $X_S$. Let $\beta$ be the
  Busemann function associated to the end $\mathbf{P}$. By the
  previous paragraph $\Gamma$ acts cocompactly on $\beta^{-1}(I)$ for
  any compact interval $I\subset \R$ because it has finite index in
  $\mathbf{P}(\mco_S)$. Since $\Gamma$ has only $p$-torsion and its
  action is proper, Lemma \ref{l:horoball-semiduality} applies.
\end{proof}

Suppose $F$ is a finite group. The {\em lamplighter group with base
  group $F$} is
$\Gamma_F = F\wr \Z = \left( \oplus_{i\in \Z} F \right) \rtimes \Z$,
where $\Z$ acts by shifting the indices of a sequence $(f_i)$. 

\begin{theorem} \label{thm:lamplighter} The lamplighter group with
  base group $F$ is a $\Z[1/\abs{F}]$-semiduality group of dimension
  2.
\end{theorem}

\begin{proof}
  Let $T_1$ and $T_2$ be copies of a $(\abs{F}+1)$-regular tree. The
  lamplighter group $\Gamma_F$ acts on $T_1\times T_2$ in a natural way;
  for description of the action see \cite[\S4]{wortmanfp}. This action
  preserves a Busemann function $\beta$ and is cocompact on any set of the
  form $\beta^{-1}(I)$ for closed intervals $I\subseteq \R$. Stabilizers
  of cells are finite sums of copies of $F$. Therefore Lemma
  \ref{l:horoball-semiduality} applies.
\end{proof}

There are ``higher rank'' generalizations of lamplighter groups known
as {\em Diestel-Leader groups} $\Gamma_d(q)$ which act on a product of
$d$ regular trees of valence $q+1$. These are constructed in
\cite{MR2421161} for any values of $d$ and $q$ such that $d\leq p+1$
for any prime $p$ dividing $q$; a lamplighter group with base group
$F$ is an example of $\Gamma_2(\abs{F})$. The proof of Theorem
\ref{thm:lamplighter} easily generalizes to prove:

\begin{theorem}
  A Diestel-Leader group $\Gamma_d(q)$ is a $\Z[1/q]$-semiduality
  group of dimension $d$.
\end{theorem}

As a final remark, consider a countable collection of finite groups
$\{ F_i \}_{i\in \N}$ and let $\Lambda = \oplus_{i\in \N} F_i$. (This
is not necessarily solvable.) Then $\Lambda$ is an $R$-semiduality
group of dimension 1 for any principal ideal domain $R$ in which
$\abs{F_i}$ is invertible for every $i$. (So, for example, any
countable sum of finite groups is a $\Q$-semiduality group.) To see
this, let $\Lambda_n = \oplus_{i=0}^n F_i$. Form a graph of groups
with underlying graph a simplicial ray whose $n$th vertex and
proceeding edge are labeled by $\Lambda_n$, with inclusion maps from
edge groups to incident vertex groups. Then $\Lambda$ is the
fundamental group of this graph of groups. It acts on the Bass-Serre
tree preserving a height function inherited from the base ray, and is
cocompact on preimages of closed intervals. Cell stabilizers are
isomorphic to some $\Lambda_n$, so Lemma \ref{l:horoball-semiduality}
produces the desired result.


\end{document}